\documentclass[a4paper,12pt]{amsart}
\usepackage{amsmath,amssymb,amsfonts,amsthm}

\usepackage[alphabetic]{amsrefs}
\usepackage{bm}

\usepackage[dvipdfmx]{graphicx}
\usepackage {seqsplit}
\usepackage{ascmac}
\usepackage[all]{xy}
\usepackage{amsthm}
\usepackage[top=25mm,bottom=25mm,left=25mm,right=25mm]{geometry}
\usepackage{tikz}
\usetikzlibrary{cd,arrows,matrix,backgrounds,shapes,positioning,calc,decorations.markings,decorations.pathmorphing,decorations.pathreplacing}
\tikzset{blackv/.style={circle,fill=black,inner sep=3pt,outer sep=3pt},
         whitev/.style={circle,fill=white,draw=black,inner sep=3pt,outer sep=3pt},
         blabel/.style={circle,draw=black,inner sep=1.5pt,outer sep=0pt},
         redv/.style={circle,fill=red,inner sep=3pt,outer sep=3pt},
         block/.style={draw,rectangle split,rectangle split horizontal,rectangle split parts=#1},
         symbol/.style={
           draw=none,
           every to/.append style={
             edge node={node [sloped, allow upside down, auto=false]{$#1$}}}}
}
\usepackage{multirow}
\newtheorem{theorem}{Theorem}[section]
\usepackage{mathtools,leftidx}

\newtheorem{proposition}[theorem]{Proposition}

\newtheorem{conjecture}[theorem]{Conjecture}
\newtheorem{corollary}[theorem]{Corollary}
\newtheorem{lemma}[theorem]{Lemma}

\theoremstyle{definition}
\newtheorem{remark}[theorem]{Remark}
\newtheorem{example}[theorem]{Example}
\newtheorem{definition}[theorem]{Definition}

\newtheorem{question}[theorem]{Question}

\makeatletter
 
 \@addtoreset{equation}{section}
\makeatother

\def\TT{\mathbb{T}}

\def\ZZ{\mathbb{Z}}
\def\QQ{\mathbb{Q}}

\title[Uniqueness theorem of generalized Markov numbers that are prime powers]{Uniqueness theorem of generalized Markov numbers that are prime powers}
\author{Yasuaki Gyoda}
\author{Shuhei Maruyama}
\keywords{Cohn triple, Cohn tree, generalized Markov equation, Markov number, Uniqueness conjecture}
\subjclass[2020]{11D25, 11J60}
\address{Graduate School of Mathematical Sciences, The University of Tokyo, 3-8-1 Komaba Meguro-ku Tokyo 153-8914, Japan}
\email{gyoda-yasuaki@g.ecc.u-tokyo.ac.jp}
\address{School of Mathematics and Physics, College of Science and Engineering, Kanazawa University, Kakuma-machi Kanazawa-shi Ishikawa 920-1192, Japan}
\email{smaruyama@se.kanazawa-u.ac.jp}
\begin{document}
\begin{abstract}In this paper, we study positive integer solutions to a generalized form of the Markov equation, given as $x^2 + y^2 + z^2 + k(yz + zx + xy) = (3 + 3k)xyz$. This equation extends the classical Markov equation $x^2 + y^2 + z^2 = 3xyz$.  We generalize the concept of Cohn triples for the classical Markov equation to the generalized Markov equations. Using this, we provide a generalization of the uniqueness theorem of Markov numbers that are prime powers.
\end{abstract}
\maketitle
\section{introduction}
\subsection{Backgrounds}
The \emph{Markov equation}, 
\[x^2+y^2+z^2=3xyz,\] was studied by Markov in the late 1870s in the context of the Diophantine approximation \cites{mar1,mar2}. Every positive integer solution (up to order) of this equation, which is called a $\emph{Markov triple}$, can be obtained by applying the $\emph{Vieta jumpings}$
\[(a,b,c)\mapsto \left(a,\dfrac{a^2+b^2}{c},b\right)\quad \text{or} \quad (a,b,c)\mapsto \left(b,\dfrac{b^2+c^2}{a},c\right)\] to $(1,1,1)$ repeatedly. A number appearing in Markov triples is called a \emph{Markov number}. Due to the discoveries of various mathematicians, this combinatorial structure of Markov triples is known today to appear in a variety of theories, including combinatorics, continued fractions, hyperbolic geometry and cluster algebra theory (in detail, see \cite{aig}). For this equation, the following conjecture was proposed by Frobenius in 1913: 

 \begin{conjecture}[the uniqueness conjecture for Markov numbers, \cite{frobenius}]\label{conj:markov}
 For any Markov number $b$, there is a unique Markov triple $(a,b,c)$ up to order such that $\max\{a,b,c\}=b$.
 \end{conjecture}

 Conjecture \ref{conj:markov} has been solved affirmatively in special cases, but it has not been fully solved yet. One known approach to Conjecture \ref{conj:markov} is to use the following criterion:

\begin{theorem}[see \cite{aig}*{Corollary 3.17}]\label{thm:Markov-criterion-intro}
For a Markov number $b\neq 1$ or $2$, if the number of solutions to $x^2+ 1\equiv 0 \mod b$ is at most two, then Conjecture \ref{conj:markov} is true for $b$.
\end{theorem}

By using this criterion, the following theorem is proved:

\begin{theorem}[see \cite{aig}*{Corollary 3.20}]\label{cor:markov-p^m}
For a Markov number $b$, if there exists an odd prime $p$ and $m\in \mathbb Z_{\geq 1}$ such that $b=p^m$ or $2p^m$, then Conjecture \ref{conj:markov} is true for $b$.
\end{theorem}

Theorem \ref{cor:markov-p^m} has been (partially or fully) proved by some mathematicians using various methods other than the criterion, for example, Baragar \cite{baragar}, Button \cite{button}, Schmutz \cite{schmutz}, Zhang \cite{zhang}, Lang--Tan \cite{lang-tan} and so on.

Recently, a generalization of the Markov equation endowed with the combinatorial structure of positive integer solutions was given in relation to cluster algebra theory by Banaian \cite{bana}, the first author \cite{gyo21}, and the first author and Matsushita \cite{gyomatsu}. In generalization in this sense, the most generalized class currently known is the following equation given by \cite{gyomatsu}:
 \begin{align}\label{gm-generalized-eq}
    x^2+y^2+z^2+k_1yz+k_2zx+k_3xy=(3+k_1+k_2+k_3)xyz,
 \end{align}
 where $k_1,k_2,k_3 \in \mathbb Z_{\geq 0}$. In papers \cites{gyo21,bansen,gyomatsu}, they attempted to extend the theory of classical Markov equation to the theory of these generalized equations. In this context, Conjecture \ref{conj:markov} gives the following question:

 \begin{question}\label{ques:intro}
 For any number $b$ appearing in a positive integer solution to \eqref{gm-generalized-eq}, is there a unique positive integer solution $(a,b,c)$ up to order such that $\max\{a,b,c\}=b$?
 \end{question}

The answer to this question is \emph{no} if $k_1,k_2,k_3$ are not the same number, and a counterexample is given by \cite{gyomatsu}. However, any counterexamples have not been found yet in the case where $k_1=k_2=k_3$.

 \subsection{Main results} 
 
 In this paper, we consider Question \ref{ques:intro} for the case where $k_1=k_2=k_3$. The equation 
 \[x^2+y^2+z^2+k(yz+zx+xy)=(3+3k)xyz\]
obtained by substituting $k=k_1=k_2=k_3$ to \eqref{gm-generalized-eq} is called the \emph{$k$-generalized Markov equation}. We abbreviate this equation as $\mathrm{GME}(k)$.
 
As is the classical case, a positive integer solution (up to order) of this equation is called the $\emph{$k$-generalized Markov triple}$. According to \cite{gyomatsu}, every $k$-generalized Markov number can be obtained by applying the $\emph{Vieta jumpings}$
\[(a,b,c)\mapsto \left(a,\dfrac{a^2+kab+b^2}{c},b\right)\quad \text{or} \quad (a,b,c)\mapsto \left(b,\dfrac{b^2+kbc+c^2}{a},c\right)\] to $(1,1,1)$ repeatedly. A number appearing in $k$-generalized Markov triples is called a \emph{$k$-generalized Markov number}.

 One of the main results in this paper is a generalization of Theorem \ref{thm:Markov-criterion-intro}:

\begin{theorem}\label{thm:criterion-intro}
Let $k\in \mathbb Z_{\geq 0}$. For a $k$-generalized Markov number $b\neq 1$ or $k+2$, if the number of solutions to $x^2+ kx+1\equiv 0 \mod b$ is at most two, then there is a unique $k$-generalized Markov triple $(a,b,c)$ up to order such that $\max\{a,b,c\}=b$.
\end{theorem}

 By using this criterion, we will prove the following theorem: 

 \begin{theorem}\label{thm:intro1}
For any $k\in \mathbb Z_{\geq 0}$, if $b=p$ or $2p$, where $p$ is a prime, then there a unique $k$-generalized Markov triple $(a,b,c)$ up to order such that $\max\{a,b,c\}=b$.
 \end{theorem}

 This is the $k$-generalized version of Theorem \ref{cor:markov-p^m} for the case where $m=1$. For the case where $m>1$, we can extend Theorem \ref{cor:markov-p^m} to the following form:
 
 \begin{theorem}\label{thm:intro2}
For a $k$-generalized Markov number $b$, if there exists a prime $p$ and $m\in \mathbb Z_{\geq 2}$ such that $b=p^m$ or $2p^m$ and either of the following three conditions about $k$ is satisfied, then there is a unique $k$-generalized Markov triple $(a,b,c)$ up to order such that $\max\{a,b,c\}=b$:  
\begin{itemize}
\item [(1)] $k=2$, 
\item [(2)] $k\geq 4$ is even, and both $\dfrac{k}{2}+1$ and $\dfrac{k}{2}-1$ are not divided by $p^2$, and
\item [(3)] $k$ is odd, and both $k+2$ and $k-2$ are not divided by $p^2$.
\end{itemize}
 \end{theorem}
Of the non-negative integers less than $100$, \begin{align*}
k=&1, 2,  3, 4, 5, 8, 9, 12, 13, 15, 17, 19, 21, 24, 28, 31, 32, 33, 35, 36, 37, 39, 40, 41, 44, 45, 49, 53, \\
&55, 57, 59, 60, 63, 64, 67, 68, 69, 71, 72, 75, 76, 80, 81, 84, 85, 87, 89, 91, 93, 95, 99
\end{align*} satisfy the assumption of Theorem \ref{thm:intro2} for every $p$.

From the above theorems, it seems that Question \ref{ques:intro} is a problem worth considering when $k=k_1=k_2=k_3$.

 \begin{conjecture}\label{conj:intro}
 Let $k\in \mathbb Z_{\geq 0}$. For any $k$-generalized Markov number $b$, there is a unique $k$-generalized Markov triple $(a,b,c)$ up to order such that $\max\{a,b,c\}=b$.
 \end{conjecture}

 \subsection{Generalization of Cohn triple}
In 1950s, the \emph{(classical) Cohn triple} was introduced by Cohn \cite{cohn} in his study of the Markov equation. In the present paper, we will introduce a generalization of it. They play an important role on the proof of Theorems \ref{thm:criterion-intro}--\ref{thm:intro2}. 

A classical Cohn triple $(P,Q,R)=\left(\begin{bmatrix}    p_{11} &p_{12}\\p_{21}&p_{22}
\end{bmatrix},\begin{bmatrix}    q_{11} &q_{12}\\q_{21}&q_{22}
\end{bmatrix},\begin{bmatrix}    r_{11} &r_{12}\\r_{21}&r_{22}
\end{bmatrix}\right)$ is a triple of elements of $SL(2,\mathbb Z)$ which satisfies
 \begin{itemize}
     \item $(p_{12},q_{12},r_{12})$ is a Markov triple,
     \item $Q=PR$,
     \item $(\mathrm{tr} (P),\mathrm{tr}(Q),\mathrm{tr}(R))=(3p_{12},3q_{12},3r_{12})$.
\end{itemize}
The Cohn triple is a ``matrixization" of the Markov triple. Indeed, the following theorem holds:
\begin{theorem}[see \cite{aig}*{Theorem 4.7}]
Every Cohn triple $(a,b,c)$ with $b\geq \max\{a,c\}$ can be obtained by applying 
\begin{equation}\label{eq:cohn-jumping}
(P,Q,R)\mapsto (P,PQ,Q)\quad \text{or} \quad(P,Q,R)\mapsto (Q,QR,R)
\end{equation}
successively to a Cohn triple with $(p_{12},q_{12},r_{12})=(1,1,1)$. Moreover, the transformations of $(1,2)$-entries in \eqref{eq:cohn-jumping} coincide with the Vieta jumpings of the Markov triples. 
\end{theorem}

We will explain how to generalize the Cohn triples to triples which are compatible with the $k$-generalized Markov equation. It is known that a triple $(p_{12},q_{12},r_{12})$ is a Markov triple if and only if $(3p_{12},3q_{12},3r_{12})$ is a positive integer solution to the \emph{second Markov equation},
\[x^2+y^2+z^2=xyz\]
(see \cite{aig}*{Chapter 2}).
We introduce the \emph{$k$-generalized second Markov equation},
\[x^2+y^2+z^2+(k^2+2k)(x+y+z)+2k^3+3k^2=xyz.\]
We abbreviate the equation as $\textrm{GSME}(k)$. 
If a triple $(p_{12},q_{12},r_{12})$ is a $k$-generalized Markov triple, then \[((3+3k)p_{12}-k,(3+3k)q_{12}-k,(3+3k)r_{12}-k)\] is a positive integer solution to $\textrm{GSME}(k)$ by a straightforward calculation. Based on this, we define the \emph{$k$-generalized Cohn triple} as a triple $(P,Q,R)$ which satisfies
\begin{itemize}
     \item $(p_{12},q_{12},r_{12})$ is a $k$-generalized Markov triple,
     \item $Q=PR-S$, where $S=\begin{bmatrix}
         k&0\\3k^2+3k&k
     \end{bmatrix}$,
     \item $(\mathrm{tr} (P),\mathrm{tr}(Q),\mathrm{tr}(R))=((3+3k)p_{12}-k,(3+3k)q_{12}-k,(3+3k)r_{12}-k)$,
\end{itemize}
and we prove the following theorem:
\begin{theorem}\label{thm:cohn-intro}
Every $k$-generalized Cohn triple $(a,b,c)$ with $b\geq \max\{a,c\}$ is obtained by applying 
\begin{equation}\label{eq:k-gen-cohn-jumping}
    (P,Q,R)\mapsto (P,PQ-S,Q)\quad \text{or} \quad (P,Q,R)\mapsto (Q,QR-S,R)
\end{equation}
successively to a $k$-generalized Cohn triple with $(p_{12},q_{12},r_{12})=(1,1,1)$. Moreover, the transformations of $(1,2)$-entries in \eqref{eq:k-gen-cohn-jumping} coincide with the Vieta jumpings of the $k$-generalized Markov triples. 
\end{theorem}
The generalized second Markov equation is derived from  the following identity for any matrices $A, B, C$ in $SL(2,\mathbb C)$, which have been found by Luo \cite{luo} and N\"a\"at\"anen--Nakanishi \cite{nana}: for $a:=-\mathrm{tr}(A), b:=-\mathrm{tr}(B),c:=-\mathrm{tr}(C), d:=-\mathrm{tr}(ABC), x:=-\mathrm{tr}(AB), y:=-\mathrm{tr}(BC), z:=-\mathrm{tr}(CA)$, the equality
\[x^2+y^2+z^2+(ad+bc)x+(bd+ca)y+(cd+ab)z+a^2+b^2+c^2+d^2+abcd-4=xyz\]holds.
For triples $(A,B,C)$ and $(P,Q,R)$ which satisfy 
\begin{itemize}
    \item $\mathrm{tr}(A)=\mathrm{tr}(B)=\mathrm{tr}(C)=-k$,
    \item $\mathrm{tr}(ABC)=-2$, 
    \item $\mathrm{tr}(P)=-\mathrm{tr}(BC),\ \mathrm{tr}(Q)=-\mathrm{tr}(CA),\ \mathrm{tr}(R)=-\mathrm{tr}(AB),$
\end{itemize}$(\mathrm{tr}(P),\mathrm{tr}(Q),\mathrm{tr}(R))$ is a solution to the $k$-generalized second Markov equation.  It turns out that a triple $(A,B,C)$ of $SL(2,\mathbb Z)$ satisfying the above conditions can be taken for each $k$-generalized Cohn triple $(P,Q,R)$. Relations between this triple $(A,B,C)$ and the $k$-generalized Cohn triple $(P,Q,R)$ will be studied in a forthcoming paper.
\subsection*{Acknowledgements}
The authors would like to thank Toshiki Matsusaka for helpful comments. The first author is supported by JSPS KAKENHI Grant Number 22KJ0731, and the second author is partially supported by JSPS KAKENHI Grant Number 23KJ1938, 23K12971.
\section{Generalized Markov equation}
We consider the following binary tree $\mathrm{WM}\mathbb T(k)$:
\begin{itemize}
\item [(1)] the root vertex is $(1,1,1)$,
\item [(2)] every vertex $(a,b,c)$ has following two children;
\[\begin{xy}(0,0)*+{(a,b,c)}="1",(30,-15)*+{\left(b,\dfrac{b^2+kbc+c^2}{a},c\right).}="2",(-30,-15)*+{\left(a,\dfrac{a^2+kab+b^2}{c},b\right)}="3", \ar@{-}"1";"2"\ar@{-}"1";"3"
\end{xy}\]
\end{itemize}
The tree $\mathrm{WM}\mathbb T(k)$ is called the \emph{wide $k$-generalized Markov tree}. The transformations
\[(a,b,c)\mapsto \left(a,\dfrac{a^2+kab+b^2}{c},b\right) \quad \text{and}\quad (a,b,c)\mapsto \left(a,\dfrac{a^2+kab+b^2}{c},b\right)\] are called the \emph{left Vieta jumping} and the \emph{right Vieta jumping}, respectively. 
Note that for every $(a,b,c) \in \mathrm{WM}\mathbb T(k)$, $b$ is the unique maximal number in $a,b$ and $c$ if $(a,b,c)$ is not the root vertex. The following  proposition is the special case of \cite{gyomatsu}*{Theorem 1}:
\begin{proposition}\label{prop:all-markov}
All vertices in $\mathrm{WM}\mathbb{T}(k)$ are $k$-generalized Markov triples, and each $k$-generalized Markov triple $(a,b,c)$  with $b> \max\{a,c\}$ appears exactly twice in $\mathrm{WM}\mathbb T(k)$. Moreover, one of them appears in the left descendants of the root, and the other does in the right descendants of the root.
\end{proposition}

Since there are vertices in $\mathrm{WM}\mathbb{T}(k)$ assigned the same triple, we take the full subtree with the left child of $\mathrm{WM}\mathbb{T}(k)$ as the root to resolve this. This tree is called \emph{$k$-generalized Markov tree} and is denoted by $\mathrm{M}\mathbb{T}(k)$. By Proposition \ref{prop:all-markov}, we have the following corollary:

\begin{corollary}\label{cor:all-markov}
All vertices in $\mathrm{M}\mathbb{T}(k)$ are $k$-generalized Markov triples, and each $k$-generalized Markov triple $(a,b,c)$  with $b> \max\{a,c\}$ appears exactly once in $\mathrm{M}\mathbb T(k)$.
\end{corollary}

By using the $k$-generalized Markov equation, we have
\begin{align*}
\dfrac{a^2+kab+b^2}{c}&=(3+3k)ab-c-k(a+b),\\
\dfrac{b^2+kbc+c^2}{a}&=(3+3k)bc-a-k(b+c),
\end{align*}
and it is another presentation of switched elements.

The following is a basic property of the $k$-generalized Markov triple:

\begin{proposition}[\cite{gyomatsu}*{Corollary 8}]\label{relatively-prime}
For every $k$-generalized Markov triple $(a,b,c)$, all pairs in $a,b$ and $c$ are relatively prime.
\end{proposition}

In this subsection, we discuss the relation between $\mathrm{GME}(k)$ and $\mathrm{GSME}(k)$. By a straightforward calculation, we have the following proposition:
\begin{proposition}\label{pr:rationalsolution}
A triple $(a,b,c)$ is one of the rational solutions to $\mathrm{GME}(k)$ if and only if the triple
\[((3+3k)a-k,(3+3k)b-k,(3+3k)c-k)\] is one of the rational solutions to $\mathrm{GSME}(k)$.     
\end{proposition}
By Proposition \ref{pr:rationalsolution}, if $(a,b,c)$ is a positive integer solution to GME$(k)$, then $((3+3k)a-k,(3+3k)b-k,(3+3k)c-k)$ is a positive integer solution to GSME$(k)$. In the case where $k=0$, the converse holds (see \cite{aig}*{Proposition 2.2}), but in general, this does not hold. 

\begin{example}
We set $k=4$. Then $(9,9,22)$ is a positive integer solution to $\mathrm{GSME}(4)$, but the corresponding solution $\left(\dfrac{13}{15},\dfrac{13}{15},\dfrac{26}{15}\right)$ to $\mathrm{GME}(4)$ is not an integer triple.
\end{example}

In parallel with $\mathrm{GME}(k)$, there is an algorithm to obtain a new positive integer solution to $\mathrm{GSME}(k)$ from a given solution.

\begin{proposition}\label{pr:vietajumpGSME}
Let $(a,b,c)$ be an integer solution to $\mathrm{GSME}(k)$. Then
\[(a,ab-c-k^2-2k,b)\quad \text{and}\quad  (b,bc-a-k^2-2k,c)\]
are also integer solutions to $\mathrm{GSME}(k)$.
\end{proposition}
\begin{proof}
We only prove that $(a,ab-c-k^2-2k,b)$ is an integer solution. By Proposition \ref{pr:rationalsolution}, $\left(\dfrac{a+k}{3+3k},\dfrac{b+k}{3+3k},\dfrac{c+k}{3+3k}\right)$ is a rational solution to $\mathrm{GME}(k)$. Applying the left Vieta jumping of $\mathrm{GME}(k)$, we have a rational solution to $\mathrm{GME}(k)$,
$\left(\dfrac{a+k}{3+3k},\dfrac{ab-c-k^2-k}{3+3k},\dfrac{b+k}{3+3k}\right)$. By Proposition \ref{pr:rationalsolution} again, $(a,ab-c-k^2-2k,b)$ is a rational solution to $\mathrm{GSME}(k)$. Clearly, it is an integer solution. 
\end{proof}

A positive integer solution $(a,b,c)$ to $\mathrm{GSME}(k)$ is called an \emph{induced (positive) solution (from $\mathrm{GME}(k)$)} if $\left(\dfrac{a+k}{3+3k},\dfrac{b+k}{3+3k},\dfrac{c+k}{3+3k}\right)$ is a positive integer triple.
\begin{proposition}
Let $(a,b,c)$ be an induced solution to $\mathrm{GSME}(k)$. Then
\[(a,ab-c-k^2-2k,b)\quad \text{and}\quad  (b,bc-a-k^2-2k,c)\]
are also induced solutions to $\mathrm{GSME}(k)$.    
\end{proposition}
\begin{proof}
We only prove the case where $(a, ab-c-k^2-2k,b)$. It suffices to show that $ab-c-k^2-2k$ is positive. Since $\dfrac{ab-c-k^2-k}{3+3k}\geq 1$, we have $ab-c-k^2-k\geq 3+3k$. Therefore, we have $ab-c-k^2-2k\geq 3+2k\geq 1$.
\end{proof}

In parallel with $\mathrm{GME}(k)$, the transformations
\[(a,b,c)\mapsto (a,ab-c-k^2-2k,b),\quad (a,b,c)\mapsto(b,bc-a-k^2-2k,c)\]
are called the \emph{left Vieta jumping} and the \emph{right Vieta jumping}, respectively.

We denote by $\mathrm{SM}\mathbb T(k)$ the tree obtained from $\mathrm{M}\mathbb T(k)$ by replacing $(a,b,c)$ with $((3+3k)a-k,(3+3k)b-k,(3+3k)c-k)$. From the above discussion, we have the corollary:

\begin{corollary}
All vertices in $\mathrm{SM}\mathbb{T}(k)$ are induced solutions to $\mathrm{GSME}(k)$, and each induced solution $(a,b,c)$ to $\mathrm{GSME}(k)$  with $b> \max\{a,c\}$ appears exactly once in $\mathrm{SM}\mathbb T(k)$. 
\end{corollary}

\section{Generalized Cohn tree}
\subsection{$k$-generalized Cohn tree}
We begin with defining the $k$-generalized Cohn matrix.
\begin{definition}
Let $k\in \mathbb {Z}_{\geq 0}$. An element $P=\begin{bmatrix}p_{11}&p_{12}\\p_{21}&p_{22}\end{bmatrix}$ of $SL(2,\mathbb Z)$ is called a \emph{$k$-generalized Cohn matrix} if
\begin{itemize}
    \item [(1)] $p_{12}$ is a $k$-generalized Markov number, and
    \item [(2)] $\mathrm{tr}(P)=(3+3k)p_{12}-k$.
\end{itemize}
\end{definition}
The equality in the condition (2) has the following presentation:
\[\mathrm{tr}(P)=\begin{bmatrix}
    3+3k & 0
\end{bmatrix}P\begin{bmatrix}
    0 \\ 1
\end{bmatrix}-k.\]
We recall the definition of $k$-generalized Cohn triple:
\begin{definition}\label{def:gen-Cohn-triple}
For $k\in \mathbb {Z}_{\geq 0}$, a triple $(P,Q,R)$ is called a \emph{$k$-generalized Cohn triple} if
\begin{itemize}
    \item [(1)] $P,Q$ and $R$ are $k$-generalized Cohn matrices,
    \item[(2)] $Q=PR-S$, where $S=\begin{bmatrix}
        k&0\\3k^2+3k &k
    \end{bmatrix}$, and
    \item[(3)] $(p_{12},q_{12},r_{12})$ is a $k$-generalized Markov triple, where $p_{12},q_{12}$ and $r_{12}$ are the $(1,2)$-entries of $P,Q$ and $R$, respectively.
\end{itemize}
The triple $(P,Q,R)$ is said to be \emph{associated with} $(p_{12},q_{12},r_{12})$.
\end{definition}

\begin{remark}
 Let $(P,Q,R)$ be a $k$-generalized Markov triple associated with $(a,b,c)$. By the definition of $k$-generalized Cohn triple, $(\mathrm{tr}(P),\mathrm{tr}(Q),\mathrm{tr}(R))$ is an induced solution to $\mathrm{GSME}(k)$ corresponding to $(a,b,c)$.     
\end{remark}

The definition of $k$-generalized Cohn matrix does not refer to its existence. We will prove that for a positive integer solution $(a,b,c)$ to $\mathrm{GME}(k)$  with $b\geq \max\{a,c\}$, there exists a $k$-generalized Cohn triple $(P,Q,R)$ such that $a=p_{12},b=q_{12}$ and $c=r_{12}$.

First, we prove the case where $(a,b,c)=(1,1,1)$. For every $\ell\in \mathbb Z$, we set
\begin{align*}
    P_{1;\ell}&=\begin{bmatrix}
        \ell&1\\-\ell^2+2k\ell+3\ell-1&-\ell+2k+3
    \end{bmatrix},\\
    Q_{1;\ell}&=\begin{bmatrix}
        k+\ell+1&1\\k^2-\ell^2+3k+\ell+1&k-\ell+2
    \end{bmatrix}, \text{ and}\\
    R_{1;\ell}&=\begin{bmatrix}
        2k+\ell+2&1\\-\ell^2-2k\ell+2k-\ell+1&-\ell+1
    \end{bmatrix}.
\end{align*}
\begin{proposition}\label{cohn-mat-with-111}
For every $\ell\in \mathbb Z$, the triple $(P_{1;\ell},Q_{1;\ell},R_{1;\ell})$ is a $k$-generalized Cohn triple. Conversely, for a $k$-generalized Cohn triple $(P,Q,R)$ satisfying $(p_{12},q_{12},r_{12})=(1,1,1)$, there exists $\ell\in \mathbb Z$ such that $(P,Q,R)=(P_{1;\ell},Q_{1;\ell},R_{1;\ell})$.
\end{proposition}
\begin{proof}
The former statement can be checked directly. Since $(P,Q,R)$ is determined by $p_{11}$ and the conditions (1)--(3) of the $k$-generalized Cohn triple, the latter statement follows (see also the proof of \cite{aig}*{Theorem 4.8}, which is the special case $k=0$ of this proposition).  
\end{proof}
For a $k$-generalized Markov triple $(a,b,c)$ with $b> \max\{a,c\}$, we prove the existence of a $k$-generalized Cohn triple corresponding to $(a,b,c)$ by using induction on the distance from $(1,1,1)$ in $\mathrm{WM}\mathbb T(k)$.
To this end, we define another binary tree, the wide \emph{$k$-generalized Cohn tree} $\mathrm{WGC}\mathbb T(k,\ell)$ for $\ell\in \mathbb Z$ as follows:
\begin{itemize}
\item [(1)] the root vertex is $(P_{1;\ell},Q_{1;\ell},R_{1;\ell})$, and
\item[(2)]every vertex $(P,Q,R)$ has the following two children;
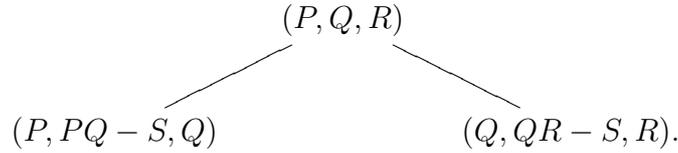
\[\begin{xy}(0,0)*+{(P,Q,R)}="1",(30,-15)*+{(Q,QR-S,R).}="2",(-30,-15)*+{(P,PQ-S,Q)}="3", \ar@{-}"1";"2"\ar@{-}"1";"3"
\end{xy}\]
\end{itemize}
We prove the following theorem:
\begin{theorem}\label{thm:cohn-markov}
Let $(a,b,c)$ be a $k$-generalized Markov triple. If $(P,Q,R)$ is a $k$-generalized Cohn triple associated with $(a,b,c)$, then $(P,PQ-S,Q)$ (resp. $(Q,QR-S,R)$) is a $k$-generalized Cohn triple associated with $\left(a,c',b\right)$ (resp. $\left(b,a',c\right)$), where $c'=\dfrac{a^2+kab+b^2}{c}$ and $a'=\dfrac{b^2+kbc+c^2}{a}$.    
\end{theorem}

To prove Theorem \ref{thm:cohn-markov}, we introduce two lemmas:

\begin{lemma}[see \cite{aig}*{Lemma 4.2}]\label{lem:basic-property-trace}
Let $A$ and $B\in SL(2,\mathbb Z)$. The following equalities hold:
\begin{itemize}
    \item [(1)] $\mathrm{tr}(A)=\mathrm{tr}(A^{-1})$,
    \item[(2)] $\mathrm{tr}(AB)= \mathrm{tr}(A)\mathrm{tr}(B)-\mathrm{tr}(AB^{-1})$, and
    \item[(3)] $A^2=\mathrm{tr}(A)A-I$.
\end{itemize}
\end{lemma}
\begin{lemma}\label{lem:M3+3kM}
For a $k$-generalized Cohn matrix $M$, the following equalities hold:
\begin{align*}
    M\begin{bmatrix}
    0&0\\ 3+3k& 0
\end{bmatrix}M&=(\mathrm{tr}(M)+k)M+\begin{bmatrix}
    0&0\\3+3k& 0
\end{bmatrix},\\
 M^{-1}\begin{bmatrix}
    0&0\\ 3+3k& 0
\end{bmatrix}M^{-1}&=-(\mathrm{tr}(M^{-1})+k)M^{-1}+\begin{bmatrix}
    0&0\\3+3k& 0
\end{bmatrix}.
\end{align*}
\end{lemma}
\begin{proof}
We only prove the first equality. We set $M=\begin{bmatrix}
    m_{11}&m_{12}\\ m_{21}& m_{22}
\end{bmatrix}$. Then we have
\begin{align*}
    M\begin{bmatrix}
    0&0\\ 3+3k& 0
\end{bmatrix}M&=M\begin{bmatrix}
    0\\1
\end{bmatrix}
\begin{bmatrix}
    3+3k& 0
\end{bmatrix}M=\begin{bmatrix}
    m_{12}\\m_{22}
\end{bmatrix}
\begin{bmatrix}
    (3+3k)m_{11}&(3+3k)m_{12}
\end{bmatrix}\\
&=(3+3k)\begin{bmatrix}
    m_{11}m_{12}&m_{12}^2\\m_{11}m_{22}&m_{12}m_{22}
\end{bmatrix}=
(3+3k)m_{12}\begin{bmatrix}
    m_{11}&m_{12}\\m_{21}&m_{22}
\end{bmatrix}+
\begin{bmatrix}
    0&0\\ 3+3k& 0
\end{bmatrix}\\
&=(\mathrm{tr}M+k)M+\begin{bmatrix}
    0&0\\3+3k& 0\end{bmatrix}.
\end{align*}
Note that $m_{11}m_{22}-m_{21}m_{12}=1$ is used to show that the fourth equal sign is valid.
\end{proof}
\begin{proof}[Proof of Theorem \ref{thm:cohn-markov}]
We only prove the statement for $(P,PQ-S,Q)$. First, we prove that the transformation $(\mathrm{tr}(P),\mathrm{tr}(Q),\mathrm{tr}(R))\mapsto(\mathrm{tr}(P),\mathrm{tr}(PQ-S),\mathrm{tr}(Q))$ is the left Vieta jumping $(\alpha,\beta,\gamma)\mapsto(\alpha,\alpha\beta-\gamma-k^2-2k,\beta)$ of $\mathrm{GSME}(k)$. Since $(P,Q,R)$ is a $k$-generalized Cohn triple, it suffices to show
\[\mathrm{tr}(PQ-S)=\mathrm{tr}(P)\mathrm{tr}(Q)-\mathrm{tr}(R)-k^2-2k.\]
Since $Q=PR-S$, we have
\begin{align*}
    \mathrm{tr}(PQ-S)&=\mathrm{tr}(P(PR-S))-2k=\mathrm{tr}(P^2R)-\mathrm{tr}(PS)-2k\\
    &\overset{\text{Lemma \ref{lem:basic-property-trace} (2)}}{=}\mathrm{tr}(P)\mathrm{tr}(PR)-\mathrm{tr}(PR^{-1}P^{-1})-\mathrm{tr}(PS)-2k\\
    &=\mathrm{tr}(P)\mathrm{tr}(PR)-\mathrm{tr}(R)-\mathrm{tr}(PS)-2k.
\end{align*}
On the other hand, since 
\[\mathrm{tr}(SP^{-1})=\mathrm{tr}\left(\begin{bmatrix}
    p_{22}k& \ast\\\ast &p_{11}k-p_{12}(3k^2+3k) \end{bmatrix}\right)=k\cdot\mathrm{tr}(P)-(3k^2+3k)p_{12}=-k^2,\] 
we have
\begin{align*}
\mathrm{tr}(P)\mathrm{tr}(Q)-\mathrm{tr}(R)-k^2-2k&=\mathrm{tr}(P)\mathrm{tr}(PR)-\mathrm{tr}(P)\mathrm{tr}(S)-\mathrm{tr}(R)-k^2-2k\\
&\overset{\text{Lemma \ref{lem:basic-property-trace} (2)}}{=}\mathrm{tr}(P)\mathrm{tr}(PR)-\mathrm{tr}(SP)-\mathrm{tr}(SP^{-1})-\mathrm{tr}(R)-k^2-2k\\
&=\mathrm{tr}(P)\mathrm{tr}(PR)-\mathrm{tr}(PS)-\mathrm{tr}(R)-2k.
\end{align*}
Therefore, we have the desired equality.
Next, we prove that $\mathrm{tr}(PQ-S)=(3+3k)m_{12}-k$, where $m_{12}$ is the $(1,2)$-entry of $PQ-S$.
By the above, we have
\begin{align*}
\mathrm{tr}(PQ-S)=&\mathrm{tr}(P)\mathrm{tr}(Q)-\mathrm{tr}(R)-k^2-2k\\
=&\left(\begin{bmatrix}
    3+3k& 0
\end{bmatrix}P\begin{bmatrix}
    0 \\ 1
\end{bmatrix}-k\right)\left(\begin{bmatrix}
    3+3k& 0
\end{bmatrix}Q\begin{bmatrix}
    0 \\ 1
\end{bmatrix}-k\right)\\
&-\left(\begin{bmatrix}
    3+3k& 0
\end{bmatrix}R\begin{bmatrix}
    0 \\ 1
\end{bmatrix}-k\right)-k^2-2k\\
=&\begin{bmatrix}
    3+3k& 0
\end{bmatrix}P\begin{bmatrix}
    0&0\\3+3k& 0
\end{bmatrix}PR\begin{bmatrix}
    0 \\ 1
\end{bmatrix}-k\begin{bmatrix}
    3+3k& 0
\end{bmatrix}PR\begin{bmatrix}
    0 \\ 1
\end{bmatrix}\\
&-k\begin{bmatrix}
    3+3k& 0
\end{bmatrix}P\begin{bmatrix}
    0 \\ 1
\end{bmatrix}-\begin{bmatrix}
    3+3k& 0
\end{bmatrix}R\begin{bmatrix}
    0 \\ 1
\end{bmatrix}-k\\
\overset{\text{Lemma \ref{lem:M3+3kM}}}{=}&\begin{bmatrix}
    3+3k& 0
\end{bmatrix}\left((\mathrm{tr}P+k)P+\begin{bmatrix}
    0&0\\3+3k& 0\end{bmatrix}\right)R\begin{bmatrix}
    0 \\ 1
\end{bmatrix}-k\begin{bmatrix}
    3+3k& 0
\end{bmatrix}PR\begin{bmatrix}
    0 \\ 1
\end{bmatrix}\\
&-k\begin{bmatrix}
    3+3k& 0
\end{bmatrix}P\begin{bmatrix}
    0 \\ 1
\end{bmatrix}-\begin{bmatrix}
    3+3k& 0
\end{bmatrix}R\begin{bmatrix}
    0 \\ 1
\end{bmatrix}-k\\
\overset{\text{Lemma \ref{lem:basic-property-trace} (3)}}{=}&\begin{bmatrix}
    3+3k& 0
\end{bmatrix}P^2R\begin{bmatrix}
    0 \\ 1
\end{bmatrix}-k\begin{bmatrix}
    3+3k& 0
\end{bmatrix}P\begin{bmatrix}
    0 \\ 1
\end{bmatrix}-k\\
=&\begin{bmatrix}
    3+3k& 0
\end{bmatrix}(P^2R-PS-S)\begin{bmatrix}
    0 \\ 1
\end{bmatrix}-k\\
=&\begin{bmatrix}
    3+3k& 0
\end{bmatrix}(PQ-S)\begin{bmatrix}
    0 \\ 1
\end{bmatrix}-k.
\end{align*}
Finally, we prove $PQ-S \in SL(2,\mathbb Z)$. We set $PQ=\begin{bmatrix}
    x_{11} & x_{12}\\ x_{21} &x_{22} 
\end{bmatrix}$.
By the above discussion, we have $\mathrm{tr}(PQ-S)=(3+3k)x_{12}-k$, and thus $\mathrm{tr}(PQ)=(3+3k)x_{12}+k$. Therefore, we have
\begin{align*}
    \det(PQ-S)&=(x_{11}-k)(x_{22}-k)-x_{12}(x_{21}-3k^2-3k)\\
    &=\det(PQ)-k\cdot\mathrm{tr}(PQ)+k^2+(3k^2+3k)x_{12}\\
    &=\det(PQ)=1.
\end{align*}
This finishes the proof.
\end{proof}

As a corollary, for a $k$-generalized Markov triple $(a,b,c)$ with $b>\max\{a,c\}$, we obtain the existence of a $k$-generalized Cohn triple associated with $(a,b,c)$.  Before describing the statement, we will introduce the \emph{canonical graph isomorphism} between two trees.

\begin{definition}
Let $\mathbb T$ and $\mathbb T'$ be full binary trees. If a graph isomorphism $f\colon \mathbb T \to \mathbb T'$ preserves the left child and the right child, then $f$ is called \emph{the canonical graph isomorphism}.
\end{definition}

\begin{corollary}\label{cor:CT-MT}
Let $\ell\in \mathbb Z$. The correspondence between $(P,Q,R)$ in $\mathrm{WGC}\TT(k,\ell)$ and $(p_{12},q_{12},r_{12})$ induces the canonical graph isomorphism between $\mathrm{WGC}\TT(k,\ell)$ and $\mathrm{WM}\TT(k)$. In particular, for every $k$-generalized Markov triple with $b\geq \max\{a,c\}$, there is a $k$-generalized Cohn matrix associated with $(a,b,c)$.
\end{corollary}

Moreover, a stronger result can be obtained.

\begin{proposition}\label{pr:all-cohn-triple}
Let $(P,Q,R)$ be a $k$-generalized Cohn triple associated with $(a,b,c)$. We assume that $b\geq \max\{a,c\}.$ Then, there exists $\ell\in \mathbb Z$ such that $(P,Q,R)$ is contained in $\mathrm{WGC}\mathbb T(k,\ell)$.    
\end{proposition}

To prove Proposition \ref{pr:all-cohn-triple}, we introduce the transformations from one vertex in $\mathrm{WGC}\mathbb T(k,\ell)$ except for the root to its parent. 
\begin{lemma}\label{Cohn-Markov2}
If $(P,Q,R)$ is a $k$-generalized Cohn triple associated with $(a,b,c)$, then $(P,R,P^{-1}(R+S))$ (resp. $((P+S)R^{-1},P,R)$) is a $k$-generalized Cohn triple associated with $\left(a,c,b'\right)$ (resp. $\left(b',a,c\right)$), where $b'=\dfrac{a^2+kac+c^2}{b}$.  
\end{lemma}

\begin{proof}
    We only prove the statement for $(P,R,P^{-1}(R+S))$. First, we prove that the transformation $(\mathrm{tr}(P),\mathrm{tr}(Q),\mathrm{tr}(R))\mapsto(\mathrm{tr}(P),\mathrm{tr}(R),\mathrm{tr}(P^{-1}(R+S)))$ is the inverse of the left Vieta jumping $(\alpha,\beta,\gamma)\mapsto(\alpha,\gamma,\alpha\gamma-\beta-k^2-2k)$ of $\mathrm{GSME}(k)$. Since $(P,Q,R)$ is a $k$-generalized Cohn matrix, it suffices to show
\[\mathrm{tr}(P^{-1}(R+S))=\mathrm{tr}(P)\mathrm{tr}(R)-\mathrm{tr}(Q)-k^2-2k.\]
Since $Q=PR-S$, we have
\begin{align*}
    \mathrm{tr}(P) \mathrm{tr}(R)- \mathrm{tr}(Q)-k^2-2k=\mathrm{tr}(P) \mathrm{tr}(R)- \mathrm{tr}(PR)-k^2.
\end{align*}
On the other hand, since $\mathrm{tr}(P^{-1}S)=-k^2$
(see the proof of Theorem \ref{thm:cohn-markov}), we have
\begin{align*}
\mathrm{tr}(P^{-1}(R+S))&=\mathrm{tr}(P^{-1}R)+\mathrm{tr}(P^{-1}S)=\mathrm{tr}(P) \mathrm{tr}(R)- \mathrm{tr}(PR)-k^2.
\end{align*}
Therefore, we have the desired equality.
Next, we prove that $\mathrm{tr}(P^{-1}(R+S))=(3+3k)m_{12}-k$, where $m_{12}$ is the $(1,2)$-entry of $P^{-1}(R+S)$.
By the above, we have
\begin{align*}
\mathrm{tr}(P^{-1}(R+S))=&\mathrm{tr}(P)\mathrm{tr}(R)-\mathrm{tr}(Q)-k^2-2k\\
=&\left(-\begin{bmatrix}
    3+3k& 0
\end{bmatrix}P^{-1}\begin{bmatrix}
    0 \\ 1
\end{bmatrix}-k\right)\left(\begin{bmatrix}
    3+3k& 0
\end{bmatrix}R\begin{bmatrix}
    0 \\ 1
\end{bmatrix}-k\right)\\
&-\left(\begin{bmatrix}
    3+3k& 0
\end{bmatrix}Q\begin{bmatrix}
    0 \\ 1
\end{bmatrix}-k\right)-k^2-2k\\
=&-\begin{bmatrix}
    3+3k& 0
\end{bmatrix}P^{-1}\begin{bmatrix}
    0&0\\3+3k& 0
\end{bmatrix}P^{-1}Q\begin{bmatrix}
    0 \\ 1
\end{bmatrix}\\
&-\begin{bmatrix}
    3+3k& 0
\end{bmatrix}P^{-1}\begin{bmatrix}
    0&0\\3+3k& 0
\end{bmatrix}P^{-1}S\begin{bmatrix}
    0 \\ 1
\end{bmatrix}\\&-k\begin{bmatrix}
    3+3k& 0
\end{bmatrix}P^{-1}Q\begin{bmatrix}
    0 \\ 1
\end{bmatrix}-k\begin{bmatrix}
    3+3k& 0
\end{bmatrix}P^{-1}S\begin{bmatrix}
    0 \\ 1
\end{bmatrix}\\
&+k\begin{bmatrix}
    3+3k& 0
\end{bmatrix}P^{-1}\begin{bmatrix}
    0 \\ 1
\end{bmatrix}-\begin{bmatrix}
    3+3k& 0
\end{bmatrix}Q\begin{bmatrix}
    0 \\ 1
\end{bmatrix}-k\\
\overset{\text{Lemma \ref{lem:M3+3kM}}}{=}&\begin{bmatrix}
    3+3k& 0
\end{bmatrix}\left((\mathrm{tr}P^{-1}+k)P^{-1}-\begin{bmatrix}
    0 & 0\\3+3k&0
\end{bmatrix}\right)Q\begin{bmatrix}
    0 \\ 1
\end{bmatrix}\\
&+\begin{bmatrix}
    3+3k& 0
\end{bmatrix}\left((\mathrm{tr}P^{-1}+k)P^{-1}-\begin{bmatrix}
    0 & 0\\3+3k&0
\end{bmatrix}\right)S\begin{bmatrix}
    0 \\ 1
\end{bmatrix}\\
&-k\begin{bmatrix}
    3+3k& 0
\end{bmatrix}P^{-1}Q\begin{bmatrix}
    0 \\ 1
\end{bmatrix}-k\begin{bmatrix}
    3+3k& 0
\end{bmatrix}P^{-1}S\begin{bmatrix}
    0 \\ 1
\end{bmatrix}\\
&+k\begin{bmatrix}
    3+3k& 0
\end{bmatrix}P^{-1}\begin{bmatrix}
    0 \\ 1
\end{bmatrix}-\begin{bmatrix}
    3+3k& 0
\end{bmatrix}Q\begin{bmatrix}
    0 \\ 1
\end{bmatrix}-k\\
\overset{\text{Lemma \ref{lem:basic-property-trace} (3)}}{=}&\begin{bmatrix}
    3+3k& 0
\end{bmatrix}(P^{-2}Q+P^{-2}S+P^{-1}S)\begin{bmatrix}
    0 \\ 1
\end{bmatrix}-k\\
=&\begin{bmatrix}
    3+3k& 0
\end{bmatrix}(P^{-1}(R+S))\begin{bmatrix}
    0 \\ 1
\end{bmatrix}-k.
\end{align*}
Note that in the second equality from the last, we use the fact that $S\begin{bmatrix}
    0\\1
\end{bmatrix}=\begin{bmatrix}
    0\\1
\end{bmatrix}$.
Finally, we prove $P^{-1}(R+S) \in SL(2,\mathbb Z)$. By setting $R=\begin{bmatrix}
    r_{11} & r_{12} \\ r_{21} &r_{22}
\end{bmatrix}$, we have
\begin{align*}
    \det(R+S)&=(r_{11}+k)(r_{22}+k)-r_{12}(r_{21}+3k^2+3k)\\
    &=\det(R)+k\cdot\mathrm{tr}(R)-k^2-(3k^2+3k)r_{12}\\
    &=\det(R)=1.
\end{align*}
This finishes the proof.
\end{proof}

\begin{proof}[Proof of Proposition \ref{pr:all-cohn-triple}]
We consider the transformation $(P, Q, R) \mapsto (P, R, P^{-1}(R+S))$ and $(P, Q, R) \mapsto ((P+S)R^{-1}, P, R)$. We assume that $(P, Q, R)$ is associated with $(a, b, c)$. If $a \leq c$ (resp., $a \geq c$), by using $(P, Q, R) \mapsto (P, R, P^{-1}(R+S))$ (resp., $(P+S)R^{-1}, P, R)$), we obtain another $k$-generalized Cohn triple $(P', Q', R')$ by Lemma \ref{Cohn-Markov2}. Let $a',b'$ and $c'$ be the $(1,2)$-entries of $P',Q'$ and $R'$, respectively. Then, we have $b' < b$. By repeating this operation, we obtain a $k$-generalized Cohn triple $(P'', Q'', R'')$ associated with $(1, 1, 1)$. According to Proposition \ref{cohn-mat-with-111}, there exists $\ell \in \mathbb{Z}$ such that $P'' = P_{1;\ell}$, $Q'' = Q_{1;\ell}$, and $R'' = R_{1;\ell}$. This implies that $(P, Q, R)$ is contained in $\mathrm{WGC}\mathbb T(k, \ell)$.
\end{proof}

\begin{remark}\label{rem:not-unique}
The integer $\ell$ in Proposition \ref{pr:all-cohn-triple} is not unique if $(a,b,c)\neq (1,1,1)$. Since there are two operations to get the $k$-generalized Cohn triple associated with $(1,1,1)$ from one associated with $(1,k+2,1)$, there are just two $\ell$'s. Therefore, $(P, Q, R)$ appears in just two wide $k$-generalized Cohn trees, one in the left half of one tree and the other in the right half of the other tree.
\end{remark}

The discussion thus far proves Theorem \ref{thm:cohn-intro}.

\begin{proof}[Proof of Theorem \ref{thm:cohn-intro}]
The statement follows from Theorem \ref{thm:cohn-markov} and Proposition \ref{pr:all-cohn-triple}.    
\end{proof}

Next, we prove the following theorem:

\begin{theorem}\label{thm:Cohn-distinct}
Let $k\in \mathbb Z_{\geq0}$ and $\ell\in \mathbb Z$. The second entries of $k$-generalized Cohn triples in $\mathrm{WGC}\mathbb T(k, \ell)$ are distinct.   
\end{theorem}

To prove Theorem \ref{thm:Cohn-distinct}, we introduce the index of a $k$-generalized Cohn matrix:

\begin{definition}
For a $k$-generalized Cohn matrix $M=\begin{bmatrix}
    m_{11}&m_{12}\\m_{21}&m_{22}
\end{bmatrix}$, the \emph{index} of $M$ is the value $I_M:=\dfrac{m_{11}}{m_{12}}$. 
\end{definition}

\begin{proof}[Proof of Theorem \ref{thm:Cohn-distinct}]
For each $(P,Q,R)$, it is enough to show that  $I_P<I_Q<I_R$. First, we prove $I_Q<I_R$. Since $Q=PR-S$ holds, we have $P=(Q+S)R^{-1}$. By comparing the $(1,2)$-entries of this equality, we have
\[
    p_{12}=r_{11}q_{12}-(q_{11}+k)r_{12}\leq r_{11}q_{12}-q_{11}r_{12}.
\]
Therefore, we have
\[0< \dfrac{p_{12}}{q_{12}r_{12}}\leq \dfrac{r_{11}}{r_{12}}-\dfrac{q_{11}}{q_{12}}=I_R-I_Q,\]
and $I_Q<I_R$ holds as desired. Next, we prove $I_P<I_Q$. We obtain $R=P^{-1}(Q+S)$ from $Q=PR-S$. By comparing the $(1,2)$-entries of this equality, we have
\begin{align*}
     r_{12}&=p_{22}q_{12}-p_{12}(q_{22}+k)\\
     &=((3+3k)p_{12}-k-p_{11})q_{12}-p_{12}((3+3k)q_{12}-q_{11})\\
     &=-kq_{12}+q_{11}p_{12}-q_{12}p_{11}\leq q_{11}p_{12}-q_{12}p_{11}.
\end{align*}
Therefore, we have
\[0< \dfrac{r_{12}}{p_{12}q_{12}}\leq \dfrac{q_{11}}{q_{12}}-\dfrac{p_{11}}{p_{12}}=I_Q-I_P,\]
and $I_P<I_Q$ holds as desired.
\end{proof}

We fix $k\in \mathbb Z_{\geq 0}$. We denote by $\mathrm{GC}\mathbb T(k,\ell)$ the full subtree of $\mathrm{WGC}\mathbb T(k,\ell)$ whose root is the left child of $(P_{1;\ell},Q_{1;\ell}, R_{1;\ell})$, and we call this subtree the \emph{$k$-generalized Cohn tree}. By the results of $\mathrm{WGC}\mathbb{T}(k,\ell)$, we have the following corollaries.

\begin{corollary}\label{cor:CT-MT2}
Let $\ell\in \mathbb Z$. The correspondence between $(P,Q,R)$ in $\mathrm{GC}\TT(k,\ell)$ and $(p_{12},q_{12},r_{12})$ induces the canonical graph isomorphism between $\mathrm{GC}\TT(k,\ell)$ and $\mathrm{M}\TT(k)$. 
\end{corollary}

\begin{corollary}\label{cor:Cohn-distinct}
Let $k\in \mathbb Z_{\geq0}$ and $\ell\in \mathbb Z$. The second entries of $k$-generalized Cohn triples in $\mathrm{GC}\mathbb T(k, \ell)$ are distinct.   
\end{corollary}

Moreover, by Proposition \ref{pr:all-cohn-triple} and Remark \ref{rem:not-unique}, the following proposition holds.

\begin{proposition}\label{pr:all-cohn-triple2}
Let $(P,Q,R)$ be a $k$-generalized Cohn triple associated with $(a,b,c)$. We assume that $b> \max\{a,c\}.$ Then, there  exist a unique  $\ell\in \mathbb Z$ and a unique vertex $v$ in $\mathrm{GC}\mathbb T(k, \ell)$ such that $v=(P,Q,R)$.    
\end{proposition}

\begin{remark}
We denote by $\mathrm{GC}\mathbb T^\ast(k,\ell)$ the full subtree of $\mathrm{WGC}\mathbb T(k,\ell)$ whose root is the right child of $(P_{1;\ell},Q_{1;\ell}, R_{1;\ell})$. In fact, $\mathrm{GC}\mathbb T^\ast(k,\ell)=\mathrm{GC}\mathbb T(k,k+\ell+1)$ holds. This fact is proved in the following way. Let $(P,Q,R)$ be a $k$-generalized triple of the root in $\mathrm{GC}\mathbb T^\ast(k,\ell)$, and $p_{11}$ be the $(1,1)$-entry of $P$. We note that $(P,Q,R)$ is a $k$-generalized Cohn triple associated with $(1,k+2,1)$. By the generation rule of $\mathrm{WGC}\mathbb T(k,\ell)$, we have $p_{11}=k+\ell+1$. By Lemma \ref{Cohn-Markov2}, $(P,R,{P}^{-1}(R+S))$ is a $k$-generalized Cohn triple associated with $(1,1,1)$. By Proposition \ref{cohn-mat-with-111}, $(P,R,{P}^{-1}(R+S))$ is uniquely determined by $p_{11}$ and we have $(P,R,{P}^{-1}(R+S))=(P_{1;k+\ell+1},Q_{k+\ell+1},R_{k+\ell+1})$. Therefore, $(P,Q,R)$ is the left child of $(P_{1;k+\ell+1},Q_{k+\ell+1},R_{k+\ell+1})$, and it is the root of $\mathrm{GC}\mathbb T(k,k+\ell+1)$. Hence we have $\mathrm{GC}\mathbb T^\ast(k,\ell)=\mathrm{GC}\mathbb T(k,k+\ell+1)$.
\end{remark}

\subsection{Farey tree and fraction labeling}

In this subsection, we introduce the Farey tree, and we label $k$-generalized Cohn matrices in $\mathrm{GC}\mathbb T(k,\ell)$ with irreducible fractions. As a preparation, we extend the concept of the relatively prime.

\begin{definition}
Non-negative integers $a$ and $b$ with $(a,b)\neq (0,0)$ are said to be \emph{relatively prime} if there are no $a'$ and $b'\in \mathbb Z_{\geq 0}$ such that $ca'=a$ and $cb'=b$ for any $c\in \mathbb Z_{>1}$, 
\end{definition}

If $a$ and $b$ are both positive integers, then the above definition is the same as the usual sense. We consider the case $a=0$. If $b=1$, then $a$ and $b$ are relatively prime, and otherwise, $a$ and $b$ are not relatively prime. 

\begin{definition}\label{irreducible-fraction}
Let $q\in\QQ_{\geq 0}\cup\{\infty\}$ and $n$ and $d\in\ZZ_{\geq 0}$. The symbol $\dfrac{n}{d}$ is called the \emph{reduced expression} of $q$ if $n$ and $d$ are relatively prime and $q = \dfrac{n}{d}$, where $\dfrac{n}{d}$ is regarded as $\infty$ when $d = 0$ and $n > 0$.
Moreover, a fraction $\dfrac{n}{d}$ is said to be \emph{irreducible} if there exists $q\in\QQ_{\geq 0}\cup\{\infty\}$ such that $\dfrac{n}{d}$ is the reduced expression of $q$.
\end{definition}

\begin{definition}
For $\dfrac{a}{b}$ and $\dfrac{c}{d}$, we denote $ad-bc$ by $\det\left(\dfrac{a}{b},\dfrac{c}{d}\right)$. A triple $\left(\dfrac{a}{b},\dfrac{c}{d},\dfrac{e}{f}\right)$ is called a \emph{Farey triple} if the following conditions hold:
\begin{itemize}
    \item [(1)] $\dfrac{a}{b},\dfrac{c}{d}$ and $\dfrac{e}{f}$ are irreducible fractions, and
    \item [(2)] $\left|\det\middle(\dfrac{a}{b},\dfrac{c}{d}\middle)\middle|=\middle|\det\middle(\dfrac{c}{d},\dfrac{d}{e}\middle)\middle|=\middle|\det\middle(\dfrac{d}{e},\dfrac{a}{b}\middle)\right|=1$.
\end{itemize}
\end{definition}

We define the \emph{Farey tree} $\mathrm{F}\mathbb T$ as follows:
\begin{itemize}
\item [(1)] the root vertex is $\left(\dfrac{0}{1},\dfrac{1}{1},\dfrac{1}{0}\right)$, and
\item[(2)]every vertex $\left(\dfrac{a}{b},\dfrac{c}{d},\dfrac{e}{f}\right)$ has the following two children;
\[\begin{xy}(0,0)*+{\left(\dfrac{a}{b},\dfrac{c}{d},\dfrac{e}{f}\right)}="1",(-30,-15)*+{\left(\dfrac{a}{b},\dfrac{a+c}{b+d},\dfrac{c}{d}\right)}="2",(30,-15)*+{\left(\dfrac{c}{d},\dfrac{c+e}{d+f},\dfrac{e}{f}\right).}="3", \ar@{-}"1";"2"\ar@{-}"1";"3"
\end{xy}\]
\end{itemize}

\begin{proposition}[see \cite{aig}*{Section 3.2}]\label{prop:property-farey}\indent
\begin{itemize}
    \item [(1)]If $\left(\dfrac{a}{b},\dfrac{c}{d},\dfrac{e}{f}\right)$ is a Farey triple, then so are $\left(\dfrac{a}{b},\dfrac{a+c}{b+d},\dfrac{c}{d}\right)$ and $\left(\dfrac{c}{d},\dfrac{c+e}{d+f},\dfrac{e}{f}\right)$. In particular, each vertex in $\mathrm{F}\mathbb T$ is a Farey triple. 
    \item [(2)] For every irreducible fraction $\dfrac{a}{b} \in \mathbb Q_{> 0}$, there exists a unique Farey triple $F$ in $\mathrm{F}\mathbb T$ such that $\dfrac{a}{b}$ is the second entry of $F$.
    \item [(3)] For $\left(\dfrac{a}{b},\dfrac{c}{d},\dfrac{e}{f}\right)$ in $\mathrm{F}\mathbb T$, $\dfrac{a}{b}<\dfrac{c}{d}<\dfrac{e}{f}$ holds. 
\end{itemize}
\end{proposition}

By using the canonical graph isomorphism from the Farey tree to the $k$-generalized Cohn tree, we provide the correspondence from Farey triple in $\mathrm{F}\mathbb T$ to $k$-generalized Cohn triple in $\mathrm{GC}\mathbb T(k,\ell)$. This correspondence induces the map from irreducible fractions in $(0,\infty)$ to $k$-generalized Cohn matrices which are the second components of $k$-generalized Cohn triples in $\mathrm{GC}\mathbb T(k,\ell)$. This map is called the \emph{fraction labeling to $k$-generalized Cohn matrices}. We denote by $C_q(k,\ell)$ the $k$-generalized Cohn matrix labeled with a fraction $q$. By Theorem \ref{thm:Cohn-distinct}, its proof and Proposition \ref{prop:property-farey} (3), we have the following corollary:

\begin{corollary}
\label{fraction-labeling-inj}
The fraction labeling to $k$-generalized Cohn matrices is injective. Moreover, if $p<q$ for $p,q\in \QQ_{\geq 0}\cup\{\infty\}$, then $I_{C_p(k,\ell)}<I_{C_q(k,\ell)}$ holds.
\end{corollary}

\section{Proof of main theorems}
In this section, we prove the criterion theorem (Theorem \ref{thm:criterion-intro}) and find $k$-generalized Markov numbers to which this theorem can be applied.

Let $\mathrm{LGC}\mathbb T(k,\ell)$ be the full subtree of $\mathrm{GC}\mathbb T(k,\ell)$ whose root is the left child of the root of  $\mathrm{GC}\mathbb T(k,\ell)$. In this section, we consider the case $\ell=-k$. In this case, the root $\left(C_{\frac{0}{1}}(k,-k),C_{\frac{1}{2}}(k,-k),C_{\frac{1}{1}}(k,-k)\right)$ of $\mathrm{LGC}\mathbb T(k,-k)$ is given by
\begin{align*}
  C_{\frac{0}{1}}(k,-k)&=\begin{bmatrix}
    -k &1\\-3k^2-3k-1&3k+3
\end{bmatrix},\\
   C_{\frac{1}{2}}(k,-k)&=\begin{bmatrix}
      k+2  &2k^2 + 6k + 5\\3k^2+9k+5 &6k^3+24k^2+31k+ 13
\end{bmatrix}, \text{ and}\\
 C_{\frac{1}{1}}(k,-k)&=\begin{bmatrix}
   1 &k+2\\3k + 2&3k^2+8k+ 5
\end{bmatrix}.
\end{align*}

All $k$-generalized Cohn matrices labeled with irreducible fractions between $0$ and $1$ are included in $\mathrm{LGC}\mathbb T(k, \ell)$. Moreover, if we restrict $\mathrm{M}\mathbb T(k)$ to the same region (we denote by $\mathrm{LM}\mathbb T(k)$ this tree), then all positive integer solutions to $\mathrm{GME}(k)$ but $(1,1,1)$ and $(1,k+2,1)$ appear exactly once without overlap (here, triples that differ only in order are regarded as the same solution). Therefore, the correspondence between the Farey triple $(r, t, s)$ and the $(1,2)$-entries of the $k$-generalized Cohn triple $(C_r(k,\ell), C_t(k,\ell), C_s(k,\ell))$ induces a bijection from the set of Farey triples in $[0,1]^3$ to the set of all $k$-generalized Markov triples but $(1,1,1)$ and $(1,k+2,1)$. If we take the second entries of $(r,t,s)$, then this bijection gives each $k$-generalized Markov number except for $1$ and $k+2$ a fraction labeling. We call it the \emph{fraction labeling to $k$-generalized Markov numbers}, and for every irreducible fraction $t\in (0,1)$, we denote by $m_t$ the corresponding $k$-generalized Markov number. Also, we set $m_\frac{0}{1}=1$ and $m_{\frac{1}{1}}=k+2$.

We will use the following proposition to show Theorem \ref{thm:criterion-intro}.

\begin{proposition}\label{prop:equivalent-condition-each}
 Let $k\in \mathbb Z_{\geq 0}$. For a $k$-generalized number $b$, there is a unique $k$-generalized Markov triple $(a,b,c)$ up to order such that $\max\{a,b,c\}=b$ if and only if there is a unique irreducible fraction $t\in [0,1]$ such that $b=m_t$. 
\end{proposition}

\begin{proof}
For $b=1$ and $b=k+2$, there exist a unique $k$-generalized Markov triple up to order such that $\max\{a,b,c\}=b$, and a unique irreducible fraction $t\in[0,1]$ such that $b=m_t$. Let $b\neq 1$ or $k+2$.  We assume that there is a unique irreducible fraction $t\in [0,1]$ such that $b=m_t$. If $m_t$ is the maximal number in two $k$-generalized Markov triples $(m_r,m_t,m_s)$ and $(m_{r'},m_t,m_{s'})$ in $\mathrm{LM}\mathbb T(k)$, then the corresponding Farey triples are $(r,t,s)$ and $(r',t,s')$. By Proposition \ref{prop:property-farey}, we have $r=r'$ and $s=s'$, and thus $m_r=m_{r'}$ and $m_s=m_{s'}$. Conversely, we assume that there is a unique $k$-generalized Markov triple $(a,b,c)$ up to order such that $\max\{a,b,c\}=b$. If $b=m_t=m_{t'}$, then by the assumption, there exist unique fractions $r$ and $s$ such that $(m_r,b,m_s)$ is a $k$-generalized Markov triples in $\mathrm{LM}\mathbb T(k)$. Since $(r,t,s)$ and $(r,t',s)$ are both Farey triple, by the definition of Farey tree, we have $t=t'$.
\end{proof}

This proposition yields the following:

\begin{corollary}\label{prop:equivalent-condition}
Conjecture \ref{conj:intro} is true if and only if the fraction labeling to the $k$-generalized Markov numbers $t\mapsto m_t$ is an injective map. 
\end{corollary}

\subsection{Proof of Theorem \ref{thm:criterion-intro}}
We fix $k\in \mathbb Z_{\geq 0}$ and a $k$-generalized Markov triple $(m_r,m_t,m_s)$ in $\mathrm{LM}\mathbb T(k)$. Note that $m_t>\max\{m_r,m_s\}$ and $m_r\neq m_s$. We consider solutions $x$ to equations
\begin{align}\label{eq:msx}
    m_rx\equiv \pm m_s \mod m_t
\end{align}
(it means that $m_rx$ is congruent to $m_s$ or $-m_s$). Since $m_r$ and $m_t$ are relatively prime from Proposition \ref{relatively-prime}, there is a unique solution to \eqref{eq:msx} in the range $\left(0,\dfrac{m_t}{2}\right)$.

\begin{definition}
 The \emph{characteristic number} $u_t$ of $(m_r,m_t,m_s)$ is a unique solution to \eqref{eq:msx} in the range $\left(0,\dfrac{m_t}{2}\right)$.
\end{definition}

\begin{lemma}\label{lem:1,1-positivity}
For an irreducible fraction $t\in (0,1)$, the $(1,1)$-entry of $C_{t}(k,-k)$ is positive.
\end{lemma}
\begin{proof}
By the monotonically increasing property of indices (Corollary \ref{fraction-labeling-inj}), it suffices to show the statement for the case $t=\dfrac{1}{n}$. First, the $(1,1)$-entry of $C_{\frac{1}{1}}(k,-k)$ is $1$. Since the $(1,1)$-entry of $C_{\frac{1}{2}}(k,-k)$ is $k+2$, it is larger than that of $C_{\frac{1}{1}}(k,-k)$. We assume that the $(1,1)$-entry of $C_{\frac{1}{i}}(k,-k)$ is positive, and that of $C_{\frac{1}{i+1}}(k,-k)$ is larger than that of $C_{\frac{1}{i}}(k,-k)$. We will prove that the $(1,1)$-entry of $C_{\frac{1}{i+2}}(k,-k)$ is larger than that of $C_{\frac{1}{i+1}}(k,-k)$, in particular, it is positive. We set \[C_{\frac{1}{i}}(k,-k)=\begin{bmatrix}
    a&b\\c&d
\end{bmatrix},\quad C_{\frac{1}{i+1}}(k,-k)=\begin{bmatrix}
    a'&b'\\c'&d'
\end{bmatrix},\quad C_{\frac{1}{i+2}}(k,-k)=\begin{bmatrix}
    a''&b''\\c''&d''
\end{bmatrix}.\]Then, we have
\begin{align*}
    C_\frac{1}{i+1}(k,-k)&=C_\frac{1}{1}(k,-k)C_\frac{1}{i}(k,-k)-S\\&=\begin{bmatrix}
    -k&1\\-3k^2-3k-1&3k+3
\end{bmatrix}\begin{bmatrix}
    a&b\\c&d
\end{bmatrix}-\begin{bmatrix}
    k&0\\3k^2+3k&k
\end{bmatrix}\\
&=\begin{bmatrix}
    -ka+c-k &\ast\\-(3k^2+3k+1)a+(3k+3)c-3k^2-3k&\ast
\end{bmatrix}. 
\end{align*}
Therefore, we have
\[a'=-ka+c-k,\quad c'=-(3k^2+3k+1)a+(3k+3)c-3k^2-3k.\]
By assumption, we have $-ka+c-k\geq a>0$. Since
\[ C_\frac{1}{i+2}(k,-k)=C_\frac{1}{1}(k,-k)C_\frac{1}{i+1}(k,-k)-S,\]
we have
\begin{align*}
    a''-a'&=-(k+1)a'+c'-k\\
    &=-(k+1)(-ka+c-k)-(3k^2+3k+1)a+(3k+3)c-3k^2-4k\\
    &=(-2k^2-2k-1)a+(2k+2)c-2k^2-3k\\
    &\geq (-2k^2-2k-1)a+(2k+2)((k+1)a+k)-2k^2-3k\\
    &=(2k+1)a-k\geq k+1>0,
\end{align*}
as desired.
\end{proof}
\begin{lemma}\label{lem:index-of-Ct}
   For an irreducible fraction $t\in (0,1)$, the following equality holds:\[C_{t}(k,-k)=\begin{bmatrix}
   u_t&m_t\\ \ast &(3+3k)m_t-u_t-k
   \end{bmatrix}.\]
\end{lemma}
\begin{proof}
 We can check the statement for $t=\dfrac{1}{2}$ directly. We set
 \[C_{t}(k,-k)=\begin{bmatrix}
   a_t&m_t\\ \ast &(3+3k)m_t-a_t-k
   \end{bmatrix},\]and we will prove $a_t=u_t$. Let $(C_r(k,-k),C_t(k,-k),C_s(k,-k))$ be a $k$-generalized Cohn triple in $\mathrm{LGC}\mathbb T(k,-k)$. Since $C_s(k,-k)=C_r(k,-k)^{-1}(C_t(k,-k)+S)$, by comparing $(1,2)$-entries, we have $m_s=a_tm_r-(a_r+k)m_t$. Therefore, we have $m_ra_t\equiv m_s \mod m_t$, which implies $a_t\equiv \pm u_t \mod m_t$. By monotonically increasing property, the inequalities $I_\frac{0}{1}<I_{t}<I_{\frac{1}{1}}$ hold, and we have $-k<\dfrac{a_t}{m_t}<\dfrac{1}{k+2}$. By Lemma \ref{lem:1,1-positivity} and the previous inequality, we have $0<a_t<\dfrac{m_t}{k+2}\leq\dfrac{m_t}{2}$. Therefore, by the definition of $u_t$, we have $a_t=u_t$.
\end{proof}

In the proof of Lemma \ref{lem:index-of-Ct}, we have the following corollary, which is a generalization of \cite{aig}*{Remark 4.14}:

\begin{corollary}
The characteristic number $u_t$ is a solution of $m_rx\equiv m_s \mod m_t$.    
\end{corollary}

Moreover, we have the following property for $u_t$:

\begin{lemma}\label{lem:solution-ut-ut'}
The characteristic number $u_t$ is a solution to $x^2+kx+1\equiv 0 \mod m_t$.  
\end{lemma}
\begin{proof}
    By $m_r^2+km_rm_s+m_s^2\equiv 0\mod m_t$ and $m_ru_t\equiv m_s \mod m_t$, we have
    \[m_r^2u_t^2+k m_r^2u_t+m_r^2\equiv 0 \mod m_t.\]
Since $m_r$ and $m_t$ are relatively prime, we have      \[u_t^2+ ku_t+1\equiv 0 \mod m_t,\]
as desired. 
\end{proof}

\begin{lemma}\label{lem:ut+k<mt/2}
 For an irreducible fraction $t\in(0,1)$, the inequality $u_t+k<\dfrac{m_t}{2}$ holds. 
\end{lemma}

\begin{proof}
    Since $m_t\geq 2k^2+6k+5$ and $u_t<\dfrac{m_t}{k+2}$ by the monotonically increasing property of the index, we have \[\dfrac{m_t}{2}-u_t>\left(\dfrac{1}{2}-\dfrac{1}{k+2}\right)m_t\geq\dfrac{k(2k^2+6k+5)}{2(k+2)}>k,\]
    as desired.
\end{proof}
Now, we are ready to prove Theorem \ref{thm:criterion-intro}.

\begin{proof}[Proof of Theorem \ref{thm:criterion-intro}]
By Proposition \ref{prop:equivalent-condition-each}, it suffices to show that a fraction $t\in (0,1)$ which satisfies $b=m_t$ is unique.
Let $t$ and $\tau$ be irreducible fractions in $(0,1)$ satisfying $b=m_t=m_{\tau}$. We will prove $t=\tau$. By Lemma \ref{lem:solution-ut-ut'}, $u_t$ and $u_\tau$ are solutions to $x^2+ kx+1\equiv 0 \mod b$. By Lemma \ref{lem:ut+k<mt/2}, we have $u_t,u_{\tau}<\dfrac{b}{2}-k$. Let $\bullet \in \{t,\tau\}$. Now, $b-(u_{\bullet}+k)$ is also a solution to $x^2+kx+1\equiv 0 \mod b$, and $b-(u_{\bullet}+k)\geq \dfrac{b}{2}$. By the assumption that the number of solutions is at most two, $u_t$ and $u_{\tau}$ are the unique solution of
\begin{align}
    x^2+kx+1\equiv 0 \mod b \text{\ and\ } 0<x<\dfrac{b}{2}.
\label{eq:unique-solution}\end{align}
Therefore, we have $u_t=u_\tau$. Moreover, by Lemma \ref{lem:index-of-Ct}, we have $C_{t}(k,-k)=C_{\tau}(k,-k)$. By Theorem \ref{thm:Cohn-distinct}, we have $t=\tau$. 
\end{proof}

\subsection{Uniqueness theorem of prime power and twice of prime power}

By Theorem \ref{thm:criterion-intro}, we can prove Theorem \ref{thm:intro1}.
To use of Theorem \ref{thm:criterion-intro}, the following lemma is essential.

\begin{lemma}\label{lem:at-most-2-solution-p}
   Let $p$ be a prime. The number of solutions to $x^2+ kx+1\equiv 0 \mod p$ is at most two.
\end{lemma}

\begin{proof}
   Let $x_1$, $x_2$ be solutions to $x^2+ kx+1\equiv 0 \mod p$ such that $x_1\not\equiv x_2$ holds. Since $(x_1-x_2)(x_1+x_2+ k)\equiv 0\mod p$, we have $x_1+x_2+ k\equiv 0\mod p$. Let $x_3$ be another solution to $x^2+ kx+1\equiv 0 \mod p$ such that $x_1\not\equiv x_3$ holds. Then, we have $x_1+x_3+ k\equiv 0\mod p$. Therefore, we have $x_2-x_3\equiv 0 \mod p$.
\end{proof}

\begin{proof}[Proof of Theorem \ref{thm:intro1}]
 When $b=4$, we have $k=2$. Since $m_t>4$ for any $t\in(0,1)$, we have a conclusion. When $b=p$, it follows from Theorem \ref{thm:criterion-intro} and Lemma \ref{lem:at-most-2-solution-p}. We assume $b=2p$ and $p\neq 2$. Clearly, $x^2+kx+1\equiv 0 \mod 2$ has at most one solution. Since $2$ and $p$ are relatively prime, by using Lemma \ref{lem:at-most-2-solution-p} and Chinese reminder theorem, $x^2+kx+1\equiv 0 \mod 2p$ has at most two solutions. Therefore, applying Theorem \ref{thm:criterion-intro}, we have the conclusion.  
\end{proof}
Next, we consider the case where $b=p^m$ or $2p^m$. In this case, using Theorem \ref{thm:criterion-intro} requires some constraints regarding $k$ and $p$.

 \begin{theorem}\label{thm:k-p-markov-conj}
For a $k$-generalized Markov number $b$, if there exists an odd prime $p$ and $m\in \mathbb Z_{\geq 2}$ such that $b=p^m$ or $2p^m$ and $k^2-4 \not\equiv 0 \mod p$, then there is a unique $k$-generalized Markov triple $(a,b,c)$ up to order such that $\max\{a,b,c\}=b$.  
\end{theorem}

It is proved by the following lemma and Theorem \ref{thm:criterion-intro}.

\begin{lemma}\label{lem:at-most-2-solutions-pm-even}
 Let $p$ be an odd prime and $m\in \mathbb Z_{\geq 2}$. If $k^2-4\not\equiv 0 \mod p$ holds, then the number of solutions to $x^2+ kx+1\equiv 0 \mod p^m$ is at most two.
\end{lemma}

\begin{proof}
Let $x_1$, $x_2$ be solutions to $x^2+ kx+1\equiv 0 \mod p^m$ such that $x_1\not\equiv x_2$ holds. Then, we have $(x_1-x_2)(x_1+x_2+ k)\equiv 0\mod p^m$. We assume that $(x_1-x_2)\equiv (x_1+x_2+ k)\equiv 0\mod p$. Then, we have $2x_1+ k\equiv 0 \mod p$. If $2x_1+ k$ is even, there exists $n\in \mathbb Z$ such that $2x_1+ k=2np$. Therefore, we have $x_1=np- \dfrac{k}{2}$. However, we have $-1\equiv x_1(x_1+ k)\equiv -\dfrac{k^2}{4}\mod p$, which conflicts with the assumption. If $2x_1+ k$ is odd, then there exists $n\in \mathbb Z$ such that $2x_1+ k=(2n+1)p$. Therefore, we have $x_1=np+ \dfrac{p- k}{2}$. However, we have $-1\equiv x_1(x_1+ k)\equiv \dfrac{p^2-k^2}{4}\mod p$, which conflicts with the assumption. Therefore, we have $x_1+x_2+ k\equiv 0\mod p^m$. Let $x_3$ be another solution to $x^2+ kx+1\equiv 0 \mod p^m$ such that $x_1\not\equiv x_3$ holds. Then, we have $x_1+x_3+ k\equiv 0\mod p^m$. Therefore, we have $x_2-x_3\equiv 0 \mod p^m$.    
\end{proof}

\begin{proof}[Proof of Theorem \ref{thm:k-p-markov-conj}]
    The cases where $b=1$ and $k+2$ are clear. We assume that $b\neq 1$ or $k+2$. The statement for $b=p^m$ follows from Theorem \ref{thm:criterion-intro} and Lemma \ref{lem:at-most-2-solutions-pm-even}. We assume $b=2p^m$.  We note that $x^2+kx+1\equiv 0 \mod 2$ has at most one solution. Since $2$ and $p^m$ is relatively prime, by Chinese reminder theorem and Lemma \ref{lem:at-most-2-solutions-pm-even}, $x^2+kx+1\equiv 0 \mod 2p^m$ has at most two solutions. Therefore, applying Theorem \ref{thm:criterion-intro}, we have the conclusion.  
\end{proof}

\begin{remark}
Theorem \ref{thm:k-p-markov-conj} is also proved by using Hensel's lifting lemma. The condition $k^2-4\not\equiv 0 \mod p$ in Theorem \ref{thm:k-p-markov-conj} is derived from the existence conditions of $x$ satisfying $f(x)=x^2+ kx+1\equiv 0 \mod p^m$ and $f'(x)=2x+ k\not\equiv 0 \mod p$. Therefore, we cannot remove this condition even if we use Hensel's lifting lemma. 
\end{remark}

In Theorem \ref{thm:k-p-markov-conj}, we only consider the case where $p$ is odd. There are no $k$-generalized Markov numbers of the form $2^m$ if $k^2-4 \not\equiv 0 \mod 2$. More strongly, the following lemma holds:

\begin{lemma}\label{lem:even-Markov}
For $k\in \mathbb Z_{\geq 0}$, if there exists an even $k$-generalized Markov number, then $k$ is even.
\end{lemma}

\begin{proof}
We will prove that if $k$ is odd, all $k$-generalized Markov number are odd. Clearly, $(1,1,1)$ is a $k$-generalized Markov triple whose elements are odd. If all of $(a,b,c)$ are odd, then  so are $(3+3k)ab-c-k(a+b)$ and $(3+3k)bc-a-k(b+c)$. Therefore, all elements of each $k$-generalized Markov triple of $\mathrm{WM}\mathbb T(k)$ are odd. By Proposition \ref{prop:all-markov}, we have the conclusion. 
\end{proof}

Furthermore, weakening the constraint on $k$ to $\dfrac{k^2}{4}-1 \not\equiv 0 \mod 2$ does not yield any additional properties:

\begin{proposition}\label{prop:4-Markov}
If there exists a $k$-generalized Markov number $b$ which is a multiple of $4$, then $k\equiv 2 \mod 4$ holds. In particular, $\dfrac{k^2}{4}-1\equiv 0 \mod 2$ holds.    
\end{proposition}

\begin{proof}
 Since $b$ is even, we have $k\not\equiv 1$ or $3\mod 4$ by Lemma \ref{lem:even-Markov}. Therefore, it suffices to show that $k\not\equiv 0 \mod 4$. We assume that $k=4n$. We set $b=4m$. We consider a $k$-generalized Cohn matrix corresponding to $4m$,
 \[C=\begin{bmatrix}
     a & 4m\\ c& (3+3(4n))\cdot 4m-a-4n
 \end{bmatrix}.\]
 Since $\det C=1$, we have
 \[1\equiv a(-a-4n) \mod 4m.\]
 If $a=2\ell$, then $a(-a-4n)$ is even, and thus $1\not\equiv a(-a-4n) \mod 4m$. If $a=2\ell+1$, then we have $a(-a-4n)\equiv -1 \mod 4$. Therefore, such a $k$-generalized Cohn matrix would not exist, which is a contradiction. 
\end{proof}

We provide a simpler sufficient condition about $p$ than Theorem \ref{thm:k-p-markov-conj}.

\begin{proposition}\label{prop:enough-large-prime}
    Let $k\in \mathbb Z_{\geq 0}$ and $p$ be an odd prime. If $k$ and $p$ satisfy either of the following two conditions, then the condition $k^2-4 \not\equiv 0 \mod p$ in Theorem \ref{thm:k-p-markov-conj} holds:
\begin{itemize}
    \item [(1)] $k\geq 4$ is even and $p>\dfrac{k}{2}+1$.
    \item [(2)] $k$ is odd and $p>k+2$.
\end{itemize}
\end{proposition}

\begin{proof}
We will prove (1). We have $\dfrac{k}{2}+1\not\equiv 0\mod p$ and $\dfrac{k}{2}-1\not\equiv 0\mod p$. Therefore, we have $\dfrac{k^2}{4}-1\not\equiv 0\mod p$. Since $p$ is an odd prime, it is equivalent to $k^2-4 \not\equiv 0\mod p$. We can prove (2) in the same way.
\end{proof}

Proposition \ref{prop:enough-large-prime} implies that if $p$ is large enough, we can apply Theorem \ref{thm:k-p-markov-conj} to $p^m$ and $2p^m$. More strongly, we have the following proposition:

\begin{proposition}\label{prop:sufficient-condition-p^m}
  Let $k\in \mathbb Z_{\geq 0}$ and $p$ be an odd prime. If $k$ and $p$ satisfy either of the following two conditions, then the condition $k^2-4 \not\equiv 0 \mod p$ in Theorem \ref{thm:k-p-markov-conj} holds:
\begin{itemize}
    \item [(1)] $k\geq 4$ is even, and both $\dfrac{k}{2}+1$ and $\dfrac{k}{2}-1$ are not divided by $p^2$,
    \item [(2)] $k$ is odd, and both $k+2$ and $k-2$ are not divided by $p^2$.
\end{itemize}
\end{proposition}
\begin{proof}
We will prove the case where $b=p^m$ (the case $b=2p^m$ is shown in exactly the same way). First, we assume that $k$ is even and $k\geq 4$. We will prove that $k^2-4\equiv 0 \mod p$ implies $\dfrac{k}{2}+1\equiv 0 \mod p^2$ or $\dfrac{k}{2}-1\equiv 0 \mod p^2$. Since $p$ is odd, we have $\dfrac{k}{2}+1\equiv 0 \mod p$ or $\dfrac{k}{2}-1\equiv 0 \mod p$. If $\dfrac{k}{2}+ 1\equiv 0 \mod p$, there exists $n\in \mathbb Z$ such that $\dfrac{k}{2}+ 1=np$. We consider the $k$-generalized Cohn matrix associated with $p^m$, $C=\begin{bmatrix}
    a &p^m \\c& (3+3k)p^m-a-k
\end{bmatrix}$. Since $\det C=1$, we have $a(-a-k)\equiv 1 \mod p$. Since $k=2np-2$, we have $-a^2- 2a\equiv 1 \mod p$. Therefore, we have $(a- 1)^2\equiv 0 \mod p$, and $a- 1\equiv 0 \mod p$. We denote by $a=n'p+ 1$. By calculating $\det C\mod p^2$, we have
\[1 \equiv -a(a+k)\equiv -(n'p+ 1)(n'p+2(np- 1)+ 1)\equiv- 2np+1 \mod p^2.\]
Hence we have $2n\equiv 0 \mod p$, and thus $n\equiv 0 \mod p$. Therefore, we have $\dfrac{k}{2}+1\equiv 0 \mod p^2$. Similarly, if $\dfrac{k}{2}-1\equiv 0 \mod p$, we have $\dfrac{k}{2}-1\equiv 0 \mod p^2$. We can prove (2) in the same way. 
\end{proof}
By using these argument, we prove Theorem \ref{thm:intro2}. The method of proof differs between (1) and (2)--(3). Let us start with (1). We use the following proposition.

\begin{proposition}[\cite{gyomatsu}*{Theorem 2.9}]\label{squre-markov}
If $(a,b,c)$ is a Markov triple, then $(a^2,b^2,c^2)$ is a $2$-generalized Markov triple. Conversely, if $(A,B,C)$ is a $2$-generalized Markov triple, then $(\sqrt{A},\sqrt{B},\sqrt{C})$ is a Markov triple.    
\end{proposition}

\begin{proof}[Proof of Theorem \ref{thm:intro2} (1)]
By Proposition \ref{squre-markov}, we have $\sqrt{b}=p^{m'}$ and it is a Markov number. By Theorem \ref{cor:markov-p^m}, there is a unique Markov triple $(a',\sqrt{b},c')$ such that $\sqrt{b}\geq \max\{a',c'\}$.  By Proposition \ref{squre-markov} again, $({a'}^2,b,{c'}^2)$ is the unique $2$-generalized Markov triple such that $b\geq \max\{{a'}^2,{c'}^2\}$.
\end{proof}

\begin{remark}
By Proposition \ref{squre-markov}, any $2$-generalized Markov number cannot be of the from $2p^m$, where $p$ is an odd prime.
\end{remark}

\begin{proof}[Proof of Theorem \ref{thm:intro2} (2) -- (3)]
The cases where $b=1$ and $k+2$ are clear. We assume that $b=1$ or $k+2$. First, we will prove the case where $p$ is an odd prime. Let $b$ be a $k$-generalized Markov number of the form $b=p^m$ or $2p^m$. By the assumption and Proposition $\ref{prop:sufficient-condition-p^m}$, we have $k^2-4\not\equiv 0\mod p$ for all odd prime $p$. Therefore, the statement follows from Theorem \ref{thm:k-p-markov-conj}. Second, we will prove that there is no $k$-generalized Markov number $b$ of the form $b=2^m$ under the conditions (2) or (3) in Theorem \ref{thm:intro2}. We assume that there is a $k$-generalized Markov number $b$ of the form $b=2^m$. By Proposition \ref{prop:4-Markov}, we have $k=4n+2$. Then, either $\dfrac{k}{2}-1$ or $\dfrac{k}{2}+1$ is divided by $4$. Therefore, $k$ satisfies neither (2) nor (3), and thus there is no $k$-generalized Markov number $b$ of the form $b=2^m$.
\end{proof}

\begin{remark}
Although Theorem \ref{thm:k-p-markov-conj} do not cover all $k$ and $p$, we have yet to find examples of $k$-generalized Markov numbers $p^m$ (or $2p^m$) to which Theorem \ref{thm:criterion-intro} cannot be applied, i.e., $k$ and $p$ satisfy that $p^m \neq1$ or $k+2$, and the equation $x^2+kx+1 \equiv 0 \mod p^m$ has more than two solutions. Even after removing the last condition, we have yet to find such an example. However, it is easy to find an example where the first condition is removed. Indeed, when we set $k=7$, $p=3$ and $m=2$ (note that $k^2-4\equiv 0 \mod 3$ holds), $9=3^2$ is the second smallest $7$-generalized Markov number, and the equation $x^2+7x+1\equiv 0 \mod 9$ has three solutions $x=1,4$, and $7$.
\end{remark}

From discussion in this section, the following theorem for general $k$-generalized Markov numbers, which is a generalization of \cite{aig}*{Theorem 3.19}, is proved.

\begin{theorem}
We assume that $k$ satisfies either of the following two conditions is satisfied:
\begin{itemize}
    \item [(1)] $k\geq 4$ is even, and both $\dfrac{k}{2}+1$ and $\dfrac{k}{2}-1$ are not divided by a squared prime,
    \item [(2)] $k$ is odd, and both $k+2$ and $k-2$ are not divided by a squared prime.
\end{itemize}
For a $k$-generalized Markov number $b=p_1^{a_1}\cdots p_n^{a_n}$ or $2p_1^{a_1}\cdots p_n^{a_n}$, where $p_i$ are odd primes, $b$ is the maximum of at most $2^{n-1}$ $k$-generalized Markov triples.
\end{theorem}

\begin{proof}
We consider how many possible fractions $t$ which satisfies $b=m_t$. To do so, we consider how many possible values we take as $u_t$, which is uniquely determined from $t$. First, we prove the case where $b=p_1^{a_1}\cdots p_n^{a_n}$. The value $u_t$ is a solution to $x^2+kx+1\equiv 0$. When $u_t$ satisfies $x^2+kx+1\equiv 0 \mod b$, it must be a solution of
\begin{align}\label{eq:x2+kx+1}
    x^2+kx+1\equiv 0 \mod p_i^{a_i}.
\end{align}
for every $p_i^{a_i}$. By Proposition \ref{prop:sufficient-condition-p^m}, for each $p_i^{a_i}$, the number of the solutions to \eqref{eq:x2+kx+1} is at most $2$. Since any two elements in $\{p_i^{m_i}\mid 1\leq i \leq n \}$ are relatively prime, by using Chinese reminder theorem, the number of solutions to $x^2+kx+1\equiv 0 \mod a$ is at most $2^n$. Of these solutions, at most $2^{n-1}$ solutions can be $u_t$ for some $t$. Otherwise, since $0<u_t<\dfrac{b}{2}-k$ and $b-(u_t+k)$ is also a solution to this equation, the number of solutions is more than $2^n$, which is a contradiction. 
Next, we will prove the case where $b=2p_1^{a_1}\cdots p_n^{a_n}$. Since the number of solutions to $x^2+kx+1\equiv 0 \mod 2$ is at most $1$, we can apply the same argument to the case $b=2p_1^{a_1}\cdots p_n^{a_n}$, and we have the conclusion.
\end{proof}

\bibliography{myrefs}

\begin{bibdiv}
\begin{biblist}

\bib{aig}{book}{
      author={Aigner, M.},
       title={Markov's {T}heorem and 100 {Y}ears of the {U}niqueness
  {C}onjecture},
   publisher={Springer},
        date={2013},
}

\bib{bana}{thesis}{
      author={Banaian, E.},
       title={{Generalizations of Cluster Algebras from Triangulated
  Surfaces}},
        type={Ph.D. Thesis},
        date={2021},
}

\bib{baragar}{article}{
      author={Baragar, A.},
       title={{On the Unicity Conjecture for Markoff Numbers}},
        date={1996},
     journal={Canad. Math. Bull.},
      volume={39},
       pages={3\ndash 9},
}

\bib{bansen}{unpublished}{
      author={Banaian, E.},
      author={Sen, A.},
       title={A generalization of markov numbers},
        date={2022},
        note={preprint, arXiv:2210.07366 [math.CO]},
}

\bib{button}{article}{
      author={Button, J.~O.},
       title={{The Uniqueness of the prime Markoff numbers}},
        date={1998},
     journal={J. London Math. Soc.},
      volume={58},
       pages={9\ndash 17},
}

\bib{cohn}{article}{
      author={Cohn, H.},
       title={{Approach to Markoff ’s minimal forms through modular
  functions}},
        date={1955},
     journal={Annals of Math.},
      volume={61},
       pages={1\ndash 12},
}

\bib{frobenius}{article}{
      author={Frobenius, G.},
       title={\"{U}ber die markovschen zahlen},
        date={1913},
     journal={Sitzungsber. Kgl. Preuss. Akad. Wiss.},
       pages={458\ndash 487},
}

\bib{gyomatsu}{article}{
      author={Gyoda, Y.},
      author={Matsushita, K.},
       title={{Generalization of Markov Diophantine equation via generalized
  cluster algebra}},
        date={2023},
     journal={Electron. J. Combin},
      volume={30},
       pages={P4.10},
}

\bib{gyo21}{unpublished}{
      author={Gyoda, Y.},
       title={Positive integer solutions to $(x+y)^2+(y+z)^2+(z+x)^2=12xyz$},
        date={2021},
        note={preprint, arXiv:2109.09639 [math.NT]},
}

\bib{lang-tan}{article}{
      author={Lang, M.~L.},
      author={Tan, S.P.},
       title={{A simple proof of the Markoff conjecture for prime powers}},
        date={2007},
     journal={Geom. Dedicata},
      volume={129},
       pages={15\ndash 22},
}

\bib{luo}{article}{
      author={Luo, F.},
       title={{Geodesic length functions and Teichm\"uller spaces}},
        date={1998},
     journal={J. Differential Geom.},
      volume={48},
       pages={275\ndash 317},
}

\bib{mar1}{article}{
      author={Markoff, A.},
       title={{Sur les formes quadratiques binaires ind\'efinies}},
        date={1879},
     journal={Math. Ann.},
      volume={15},
       pages={381\ndash 406},
}

\bib{mar2}{article}{
      author={Markoff, A.},
       title={{Sur les formes quadratiques binaires ind\'efinies (Second
  Memoir)}},
        date={1880},
     journal={Math. Ann.},
      volume={17},
       pages={379\ndash 399},
}

\bib{nana}{article}{
      author={N\"a\"at\"anen, M.},
      author={Nakanishi, T.},
       title={{Areas of two-dimensional moduli spaces}},
        date={2001},
     journal={Proc. Amer. Math. Soc.},
      volume={129},
       pages={3241\ndash 3252},
}

\bib{schmutz}{article}{
      author={Schmutz, P.},
       title={{Systoles of arithmetic surfaces and the Markoff spectrum}},
        date={1996},
     journal={Math. Ann.},
      volume={305},
       pages={191\ndash 203},
}

\bib{zhang}{unpublished}{
      author={Zhang, Y.},
       title={{An elementary proof of uniqueness of Markov numbers which are
  prime powers}},
        date={2006},
        note={preprint, arXiv:0606283 [math.NT]},
}

\end{biblist}
\end{bibdiv}
\clearpage
\appendix

\section{Lists of $k$-generalized Markov primes}

In this section, we enumerate the $k$-generalized Markov primes ($0 \leq k \leq 10$) that appear up to a depth of $10$ in the $\mathrm{LM}\mathbb T(k)$. The numbers enumerated here are those for which the uniqueness claimed in Conjecture \ref{conj:intro} holds according to Theorem \ref{thm:intro1}. Note that there may be $k$-generalized Markov primes that are smaller than numbers listed here because they are listed based on the depth of the tree.

\subsection{Case $k=0$}
There are 93 primes:

\seqsplit{$2, 5, 13, 29, 89, 233, 433, 1597, 2897, 5741, 7561, 28657, 33461, 43261, 96557, 426389, 1686049, 2922509, 3276509, 94418953, 321534781, 780291637, 1405695061, 19577194573, 25209506681, 128367472469, 143367113573, 172765826641, 311809494089, 2156735837173, 4360711162037, 31801503090601, 65082055350517, 316141040381993, 1797110092720441, 14497376284172417, 28399179102482729, 578533194662405393, 5767133969508832793, 97504568602938767093, 154466196153475428721, 3440971837880006083249, 17395856662027256694321797, 40355056064972093419107929, 91015356306777369659693437, 533085545064611653385916913, 3879105994487597861648750461, 11194070542156175027972809049, 12054029854689213198649480981, 31460683183191329406107570801, 93751461575379756817272172381, 270689299528632194363285692481, 651037040799629773892401592669, 12060176252674248098532194336069, 85467224948730400032627299618761, 599382268853352715214171250771697, 2606748241724755004225913279393673, 4082965036857519237701078947153277, 11931901168369766271363018362344337, 72591090488871527694883271387948849, 17485099331241436500304831816208615281, 71210044871026329275614448782231080593, 6996290921227261944591864320795169235553, 64174516502056990396699103795678804361073, 106920930603408183157082125869544611090157, 132309744245891382873239125373876913978917, 339175151906227045659481064270409315843901, 15615032611674039955859653729414781291177369, 37349458047631454783499863189347080654745997, 3716978032399630769202590236649048950390728014581, 466463831368434062632205434427890123606348915018273, 1497558446260381055561779253217329246757271442016229, 3012556079232600154288430154868911188481795308527949, 1945327268227131127789435205775917645151277919681930897, 2230975534285313210500248197789776594751566147528904069, 29190590062277652784355900047193160395890797208632997321, 15570385232724936425730051894737678502197799346990156608853, 106059732403157220688956131123233603686462504201395343809489, 223929840970722367725917492466048835033113404578928902340386673, 35660549109371271086347556560348399005278358369265705732178637257, 89909513192602280332673987042609143951574863326894949802919531221, 842876321548713895984657551462364447935375460135941450935868293773, 17195104666415505256099786197676708148802374715467191954581493829257, 462917397243059184474589078657761984616576912490946077800313939432633, 20533586554896767742105267798245536890140842999496538138222513180472649, 41913311503433252417405967767749933154932311777084879770785589745717121, 905909108344395064949473397052656597831780438403253233570327680503236111617, 11738224931537333296826931392405109715208951462845467660849885653556881061821, 33292141791404062267828738366793808192735303758677040659247252794628242061962162443149, 1350764071826610020640080924389722341483668796882562192645029793702857878961989087708157680721, 12542437302638098937341986562352353827202706364758443588529501994615775387815169753499594832077, 214683900183268720453485241969355139457627974763749675352266426469238282044604174730124750631489, 8233334979901891564928327532295554548294060273600257579348218922783341118390772077687903550935885447957709$.}

\subsection{Case $k=1$}
There are 67 primes:

\seqsplit{$3, 13, 61, 3673, 6673, 16693, 17857153, 27755113, 98938381, 185236633, 623860921, 20311628317, 28254272353, 86821473241, 83112623451313, 564049807419121, 3894587461664461, 61849362653982373, 217856088794091661, 3287632262847072013, 37912116519566946373, 766203028625944475790901, 23170769381747812007684353, 22812698952914776091748541501, 827935911697332396024189326317, 30483037562863384451065047559021, 59009839393143501649791937036813, 120721674464536452650143481200453, 145735708086950863798749935681461, 840926199223252533755094242157733, 849438127374094683265801105621287313, 2015366550264529819664164592842115893, 121024140926633322069833840477628268693, 2518692207195818262688340926387330650817, 10143130823942955256815901336229999647117, 121957236149326664605129521779470009256891701, 193829254218146424737157234159682106264422633, 771512134800040253792125383449688902344033033, 19633102255850612118898006127453014126085091001, 193997022379754903755415623347186524916309778693, 19359828957294756320648368469516102213658178454233, 327974301007223689343603666797846023090988563057853, 1130807625121570102116749960860818560416855384971966450961, 353058645520666328397179027475753264177546571004029887787490300581, 2615049374936460986777783535735221859028508259955251156407911566241, 9728883891459635097059635206656590745199723300635561468715856427473, 1063129637059950061235762358253138375500439065244025089114233794814613, 233610756916655207033143103417044271404129789658043248338979615959471601881, 15882026709870880781368082244488486356149541769050653584426418853677999078361, 20991527316744332999111350552680275994289588850479959284043849055106070820794817, 8130160827355932962559674308683999904846222602271940979439553292370962041938387573, 161861153961039729283702046451055004531986165801121979809443829144387980541716393186981, 263159576828068690175580535968930734282001519034734532787271603355296170351817337538017, 3734016832699647083793454602019710555159813353570969668479607206383956887989856712378650040741, 5166975315809973754552701843174483272556098927461883066399909556390450184866834211101183853933, 481298111077478663759280317851529524243105866451466841509897946131649142148873486540915655777801833881, 3547281512182831939231966119336114185231372139368673345726791915128570380172213357887958815834522687817, 66729129037325378335607241870533230855452923956989978351254937079336332058508715539826709190183361465732733, 205053415463219096559588445093793555123093488768659682312855062500994479862782753498116760412209745493110313, 118576186291566948294413337618903206778796752985580103344393773126589330456230884351935676228580100492637535321, 12317013421697130732960044748363636971547393734327540612627431496885351891821839277002209872778713135135044536847817, 101331613752390758853213670372430504968148647227408747727354970957281002609862163495960789066294723508144910080699694504003141, 37829624703839813073620708477828082720274660741472588158745825620898374156298324345914179527608879496073572439403422323648217830401, 2715094970817392625900570750421715201810932677694548141033118527738077158960547264330102711064054660154662068664060165686312157815258101, 2631451708831244181535264472505556965038822231396386529119589031855506648894277428071622913902287376205571729492823130431686672195756410992713393, 75616070543618660628128422170938269252651907213852265112848842983601764310837573215180449498186774407644449555003744749887877648074747108703605875230899163563342616317, 121399365373028214733145359431958562592007576954337453319898796323804496122719138533038704964029988452590188813965196871833549098030934455415250998615394765855655691266984611110804289761$.}

\subsection{Case $k=2$}
There is no prime, because all $2$-generalized Markov numbers are  squared numbers by Proposition \ref{squre-markov}.
\subsection{Case $k=3$}
There are 55 primes:

\seqsplit{$5, 41, 132241, 1134841, 7535401, 13873361, 19993321, 763656121, 42121344479761, 132695455574081, 4527015394166801, 4932254562930009641, 14699217590644688641, 20755070347480916201, 31730004640863679441, 3422006887000239268361, 14570010547223517780881, 79543321509059500635241, 189019839694690565532721, 1329173939074630469901688441, 19010036255861670966386724841, 30692675398583160314177328721, 7205807494050882300850852392241, 134424247566495135036068168585665001, 155158171766624512155439036493202761, 5284970879544248172251163711896269601, 194277101167156560711701525279306786921, 2581508932687984062688611015646035624324858227028281, 14861100544329388769153461768272437963984475470324230441, 846355959868352264897744309134218504610471340831973741121, 48340362591941914518817465176545190988883830106144268910921, 6537000291411228665322125478055883762958543142038845054292281, 17008128064240957217843347171516137795623493601364086647613201, 20211426879777741065690400475493317623483623291896660706290783876032881, 5690120256185911511459737131999334612593203511084445324935488230907296361, 957334822550527436504664727538069381023867877780851634895785375376489795721, 4827534988509465862606716005740268795598621302699586340677171697261316520721, 5973111340089975369011469532450932145800374682508865833338155475493478046721, 2875426248842934402970068746767028376569807652388880135149419970386165757511914561, 84820387896790891037035226222513456977270547955161181411192771461785114221187449281, 99753091036157316405702765092144940564019855976965864406462416030706925107200129961, 157705665658865858933072502608351900458677789635820740568251900179205284586683711786403521, 9010308254222215389177144286743394954873952260732999442912467659346889595283003614163112801, 3179369384000130783418864048262853383287937326418177746076517481190700524147004742584304994161, 116132621917776897082130535636186905291504702765976680998921499074069028893583193140474177604042497813561, 80959683411438137642180286616701940397908061653548803805276446347312278327605256498071029168634686158757241, 4493246392893824588532363276777023671074464588968949360502528496965771447954617993985920938826380795946752757783681, 21537067415432173162122195771887550443549384024931947953951436492245175241994752095831042939358785282687481117338593721, 6879760782001754156792403800230800049989744186764443929433319354378103857915859171696125972215966430717252727503769044576547961, 30252636620599077256344741300796106299754008280442880790397668552760450299391537948585035643914555368311828604161886580531195921, 2051609253237628688935668086874640281673208655211498198889233924481001749477253013375597499420622998647244451974519285280587316147426499329013303807601, 5377902423230703545161823865806857962350964248286341833120643280762528225172193008854294460284319274971107608276811672340289423227004624262628924940597761, 169412759681254318878406746439378769275180670205496126140491759852600420660108966965601305853165344577118502636263660275661479508339469673988193613478711393072567271937487381587274700530703241, 640497668824197831506462696355191209749402212037398909861919971604115840485455328090692686010334325725961577488101896785314988012258277915496593426710938925181512777771116399156323120817261881, 13942308088390741003916092884940303203610277395866240115850959101068510900915446253804484091125016618760268601091464941090380256498182436837250188930753159589170831054890211562591933024965206689743239168054793817721168443697793719701067430255882873708265396927090178531081$.}

\subsection{Case $k=4$}
There are 49 primes:

\seqsplit{
$61, 661, 71416021, 3319319701, 285422892301, 3947161764661, 181471304078560621, 309106080799634401, 2284112692621065162121, 16875249251984559827821, 24114712263690978573661, 21904055646392477274305881, 2606398492100174424409269481, 162951365551357498981428235921, 33797350886766706961921737909561, 30789349558523342488221584461149541, 7315936264990635956265931636951300081, 113870785728358383131842809349710274321, 870109359550086539866076332188880085521, 1774340296173487361644775870454011170621, 25552666196141012523641537199087319251541, 829352118716392386506703475855734136412954221, 46596115655825025788868272288520692903844809701, 443283599519521041764216752230880751537966183281, 2758868598942603615708939514534710746328855592321, 323089880042361055422107121860359190183007207608881, 9530346889386118538226323404314067646341159413103201, 26610577700834918432591545839165152125092080647409101, 1754083135745295450315309185894368641924760603810549561, 14241251148349048662592067647059531904383331090155223081, 6358989088449655721078330414310812712630584371009661244341803322881, 15182991857445199142037338107379514223839565982126230758823076112789421, 209970300245166602877719279719802295734728692439990654040401706957044081, 10750925408584767779726674429351297909755459031488107877886387162739919406662061, 3273451771030283883326617345006071638917267018702265967940193607831156949748893506661, 2065000982774926216030896131893585573148306718454351669086657789872330234199471001232181001, 96101031639322360633331968832560545922946407597774626090252685384129968142973484425977282737721, 2345637360265802968078362021625168968261911707427699072660358398395664437895028130638264377982255998402699468630188861, 7611526721504954901450049432442735998207543016329018607948014045159305191487146067929368093741745888101007334190932541, 11324608814419488429809201981846564953115112731647102199821180672949391792552398610940658106868133248448131359436949925481, 14559249293157774106033486071364225075126836661607597686272949997140731845363507590510288580829629570367852915767287489892192161, 131136009042828204457186861801880429873362182371457148263585999158250973969812223026928039380263585635721447069249909836885045519634461, 513505446812534948240238235704685606196014297970074839968186033371938719494525365902127050255569023121337838170703332918012853116703301, 3038044214094157129350137760066054383946601701749757891332596039014835408561599703332103805487583202574577691762095012184582090915407293684513863042596404661, 22657697993883587202614593303361576126451700615891099643787192412304106882584470610643007269789761989360438277604185037353212542749583276168929992822063083200333427061, 20389512460451220282419312201902755825880031087225117218282729671117163164248250589044548135699753694906739934562655071210589500332024104600833216677350423508253185241554212395961, 10510565334551669687959264868111511775763386931338644528473344536880756531681377335271855800971946390589439436273517051121415067393514464428353760936378118471070175177133266377534801, 7426085653561886728775381768278220021635044030378259168420534190441993035787359859281599323330907300785442415561292575966118585470473089400890572394582511855590028790011250177970544784262101, 155146934429033404929070597614589042366590520750954370277843703027706961374503709664993694314321517130144636556579995468499731050028577607260452260825803541814455782584245347546079787008727907986671331018852570661798163946715095461$.}

\subsection{Case $k=5$}
There are 38 primes:

\seqsplit{
$7, 1093, 1240009, 393723877, 36347842625557, 25428313050159271477, 49021781155904210053, 709117999268439721329097, 10505383338504940851075043333, 45797770193125248155660778581499097, 2251307000658876872062465044496730641, 98831383329340018615641396986919729889, 2870272305396763811568163812756426950859907175557, 2143384719302523478008329773599903339536562731307829, 119109791656211898671637700857068109796849120589873957994975635701, 10038574479963567248256348232525453273140627631105796165372770026962704777, 1492969571329450240929397186348089554787995336883452607032400072184957991629, 93843725652912811789995245912097442243105841861541314761251511797440390303024617, 229926861478786326195172575321269552512200964482724393305602000839726711589375634349, 32659566357126474655566313074863380212167892197956404313817392853093019366725972887644599754433, 37290774994995864379059399978028502737214629070740206938090768879288506955610532992163680876829, 6998911222606049572002910860292261182282477355382174004841529017970083562745675060410915677024498614493, 99482052218676922674979569582850475855393895854325040989390142028020614038847538045796296527514037243879221, 21678599102898627833171257483185906489352199306588321603601350742406906032364672504185012669205142195753411228364053006643902558773, 4691067995599748409314617482626031816931832792282451639961577547638542645694063958593451227601054250424889627698771787363389645591696339655601, 41530976128012158414909042923994501229739985239904450067595705767238060965068138931308699683126689532884139067811992402819039217737707042982469, 9387896741188621678335585969847099204296686948931677647839909616492834881850033115044834191832070178071335109843526596544749706479677388715295321, 85337041115008577323568999619196721984842311856993094133289106294055542568220670411621637103501519181258080027735953690062917557358833101915959437283681, 190672401990800609399104229082502195447531052617382680353858857275055743438365432501658657198767141047738450834469969455111903040446119631587558224174398466478441, 623997396539990779639638511366663448130044286918299679457673338046899908909907696610788604574980774779651335206846894006569075861447729102423187137459133434080812881, 9513241475814252516879406932561344126366005545024519830204943663357900679159189775276208038249070351778450090872700357894958985426058292174487482824665777026290684669, 4255837948001771143701527068971699143784607719047114960834190903369249977563008927463488179677717267194306328335967183121915926001060260293267500629747876168071909939889, 112514358369660080853730832339749379340352716102767475024156391333009897173395199894425452897430645160010975017651580395982236512806313059986637865754601703331982003894093, 47181464568837071005163836238316257299395991863980311612834869139456218886780951618706169967263738438783064317250032377085517469615625601451681258757225325768736979291656640706294801, 8022416362465114336475022070973907759364093851437711692989685996799978905608943925667662517087165148664375535176449143588053227406947759522587856672690014161349469885837379721067290504572190234157, 2613881758878352894527725171959461087445648708728418027817681531929209807713815292448744155255635877503115555379228688447292905833538898763777843345299116339489913801456161955685249715033078905507346232128526895860018632497877, 11314795977468686234217308821248110244481863042139420551663335477298931123503251355128636348705288566731273278375509190366785580524395990116214089618439206895231359902153384130956109268926757645991168731628137447635427087280827217, 301042428958379702595908050567901607917413719952894923581391400344105056026247304231185880590992567819228391229400750270964857036795053784252722367239863621031616395425957350817526926973353833843601586824223712093744514274688677214474563911629$.}

\subsection{Case $k=6$}
There are 37 primes:

\seqsplit{$113, 18257, 374753, 5596193, 885753233, 4955850464270161, 9820931236845233, 1154156300537819796740657, 63898167192827714492865224113, 155564470773008716820324640193, 138353516272962924749127088203019057921, 43089038055293512596755774914595638725473, 795191433127050979125774873954814176797976476099201, 93152111658726668879535480913756384350031051325764433, 3287460207282071037743081795265582606508040347151635357873, 37481840673708171705184821348361399689807608347287424279361, 633211386559733098097681563368632389943733184653600959410738940257, 100496149799768892607521734391164942612764444202924495120686605885779761, 43840384560370793177652982765446950095318647923814508751674346878701453473, 1332735494780827302651820055062177814104796359492446069710513287101212489114100593, 3786804226017567744400589141944095484865257574063819779313423404002078107029370353, 35684076109195087163719393677294180135737407471687558599877071945526736737707177761, 906482085939642768342250072857333998981449494947888948245977295839280774686618932353, 162267602402463731048968705450789093767454414450902938518947043272419699855068741447163506799713, 1476835922590919487014373386552703820239229055514698226148755301443908879097004268024462354491681, 137608669872968883222891803939747429901981719200715793670685830143024557684523741571943810795099201, 1481340873656942519525251238812427655909233932821144723819204596482051345329621012561782356048994516735270515388593, 14435038384748401710632344600244180644281173367732091591183583374713932030514117650892973106658943918957265045992913, 594352984254345029434401876185806316305172827511079271117707210502115633990888967031970139524723398173004258603749210345854775415492315872346033, 1978615748251759320172745071520740320392181370285929982664339700308747132196575664705845140824883787762238630047732631559534429328443103380502708993, 4704402369850066222630902576603550211233476068279241862313293836271239844014180386485911964252458391292618802492144327529912492949765513635721743313, 1242052321370128766723469711653672994187178905315142551745594421399790904676980948521680829474504663710952772551719829277296865471532071324781009118433214411013601, 23496624000395106996763273682308984787081902112201344805056057878975122929967576133609248286356682063729114513349521615440451727512422451407180914949088813990678248677901873, 23534303767952284525582527231110605368547802366974141545819461241230885477731876005657743549958359031079114722823593060775313808905473158886326352276757001937610314358361841, 9303582081364951481392289946147193204007083844322953010780022213646911615619742098784132394094131297652370151194118414251350488142472698072537677721290226944132700323146263547748753, 64679228542839176209047979925233466851544330417513256271183844733340953938201566482187303876722096991252248790258295500173220693377853924841509685562279765520387917483639317128459181796374357261164881980473659451173857, 12557493212604297133236320601080271239162910271779631206109002640744233663796669950328719237073303508123399834803067736979222235046164755347489784351720865694142846828351115331458078334795181373808688820239229589694991806734509574768928673742053696388595999228906406237771129459693256113$.}
\subsection{Case $k=7$}
There are 34 primes:

\seqsplit{$30241, 702721, 57694863582721, 102010032773806579921, 526679216577303259681, 141989940019695654896719681, 1230416786573088512127276241, 11089676451675798292112944561, 80045809252608730426494463009, 8048631904828380013739097910787281, 16316805876747072873785422341151883041, 351494856030828988511200064186349107329, 179007503884080521694127354653108666143041, 14724250381754052649904670616127883079878832254241, 99096082093321616256481388492039326791955862025113242321, 50411997571320203064084532712594154054184436084394253806118801, 170357023849207923876304490325952812349404678553689835015549441, 11383960909610866271600300768459306940310856711534942075509338484369, 55348156431057608637312465770335383851188085029113766584660962895321601, 173150000297829845513223128453035746679278050936615607489517902448134234643921, 157839262527155536129932470495208996666427984204063196442350833400164594162769924058289, 156966375135846001782316999106517195841368618224386023029491556787729303716327532559746478321, 4046328140867378904522474713104382641044021023115522913797508765371470137519313385199137857243281, 15533286896596988116577019625510436945030438098949311522958154637079089602678720759283912508217305147569, 158195841050527223252103539010565879755880751337685702680162244549447175499856606110414482552673935293474744174904449, 78130128117418100615914611100607341351623106816264181737392068749967165489868738288574213716182541041141474705654126881, 76671887373173134984127669354213913926230306871066174198833556407871433022876945751366425006321043347724592558263401375978414267683352296961, 487138023162715232548182323028854703437855994852990289929980402653771162499619599929018285452088540051064340961906863065515303710600806870686692379158579215761, 72539931661237784442500844497774796191648233446903724515213615736034119815398965193789415042228412904171925573952036777926489596828790408779498972299115713365969601969355041, 15160492301244251922399649931558262857393969786039272056720983778208471183518112803948870376633583421262418802911484379133778194127225330527126617968051079244313949057522591921, 2252867388373413074776927657851317895097479974412114993368547519901688859609068601655452004739065807606767430708325381910503007995399770241957564333852254921676374111381044453421969, 2228022668387567854848734998470208097588268034798815500204635019530190998228041346056024342667812013906016060328289159064814294151277784903017906794557793228850848758778758458767127823207787441, 1423813980944735723731188096300266179585908574428860370109915149152806581431313910721088152036095266226288632713194471666924553309898621118399126845001235230514521314073555352659556210577969959908454182685841514135847307202129, 68020360883741028477642896340509220979418900861043121259640539744937336857536758600395347690334092540647435297234504629029583429503605754866025742599628454740832568856002028658175806066225712191036004021630513830228076477497701798668940282462726169928366690076955447649$.}
\subsection{Case $k=8$}
There are 31 primes:

\seqsplit{$181, 3249559801, 158258325061, 397290134033101, 276929486645133061, 15311029428785850481, 377365723705932607981, 4011431270446911644101, 11513295980125316029270921899601, 141044201578443975569621904996170820181, 494961847809536156417697614010237735842881, 1007294225371319934597015278349682132381999681, 14171482031869637002384876317074552084560703401, 389950210290320052529895482171169752084332373074263211781, 5558098036596028229882858624520277461090103019085783455341, 4583365969109021927254382441139020346781346020372499310169441, 8029418335675246917417991432534083813392476663464825114792601, 12765684492116849811410304274671263318624235166366677420744386161, 8962720541749085733106027135995731964503816304085682543942786003281, 13118404930381877860965287643975932031837884987165521348034599392743887201, 37566013422100879092560639803905621576195124578290914109506184582753391971347150813921, 2870499008599459484720738508672302081688572909679858800728769446406858691151645269813331326926041, 32418659288325516312904659769387966904026415926514734495282084414412631014496717407469954734863338714991664961948187692762301, 25472050249002558160041018212231225336020792166714370098350991515826280583801918690787400345113201709621564869157001853483853181, 277231758771168323326140161277451756636401390869563871187782907338728983975136380572503886119845552628801952838793647055117319161, 92834438727094573915231893872134439701989060523107143261840022788170090073956624817340616702162255021723614442515019642559926024707139321, 32461043539317133157768320766397161936616132953106246322622557589847669144071316775493613170991973504139493797243914385007261388765052343029729005727245460805970312211143381, 5695141303923086283298831607745023568834683585252077275255078032416455951132289181805394066343539627390236952828425715566944724537330258514972243528326552932860443324110093881, 193998731609040483900840064480489569884432482371794593063315850643934930598260077828183173610460004160428821491469661476180716211200190071082491204840823736286274968732160959248483133083291053919205355861, 1281356522219913344703509603286934772632140155886553155872412428202999460187800506924792296984575532521170211884273529978131287891830712891350681336660097649150940403091678104166737077992493566523535778800281501, 3076234617525375042689350830113352774687715940429339096829788216376386193784198992280747664878000149208401793504133871644459788535549471773297730802162269131113227402053071189467468475768680972686516563614701518029588114825345040216792164172796323149366445588514454166028728933150536749543863743838781151584803129679519681$.}

\subsection{Case $k=9$}
There are 28 primes:

\seqsplit{$11, 4621, 70841, 2028401, 2343069836801, 38975976320917241, 1135329926783699021, 17109723132002248441, 3892359005608909854701, 96052546721095804900993241, 3572444154497949862834106611901, 13314399592503207645856489362041, 25506556658438710899815879651128256995244858081, 8054156759464489310267162861697064345213874703241, 1978938227361189387606715729368467315261108499832089069926162441, 310127734139115986653155845843644218764647913954368682960310408428701, 56651306872590023356832232082317174787263664931908677191922276325619321, 2691801986081999360050957215729673905029940075013628315145533743633350424957854837174137816107241, 152333881666449835601659399090114961829394212468316979834928259362041742757039532015094254898866236081, 177991630730613823339945038861027831798262650712999726380238953502056859197473254642107780603482114531183945830621, 3282708628260426575200975762009332935769583633354329345989213666057987282249465165216894200346828805396046626625364565061, 916977607580023059604499735925596565740172185300058973049233358161520262358504397713676658533832351617540962846480002918379429095919861, 22850318428481058927969399434606593663667491635144024277599094003580007784976587082057183871307387136173982568587976588418382559451448802794484424617119690119560098626729241, 3732717805535234871260518154319691969396396620682418683826018050087593854885068393023528456017540293470552134064119255961216305675055492466407662483512076695240842431692191172500904914759882839358254198312821, 76323250622110415808629707848316867698047213838179251556826334879602441722797544769660545829947915926015163896540815263354205428701199254575730869094137278059752726400179473985562961698224938128188772842661196041, 4884683516827530257155227782519652122811062307001378378800209603306059241550284189903947336384319911461513555487857295277084073696282699627057838962301845569760168532618987270388313334025589113014136657733770349918845317595320820678535357611426800428041, 35254323970493133107297797340334494177157799900463130475886524229974413275837677252557688901193591795170199414171570938698103194612396312641935919955691934972426127434507675363684722147117687132942802421058848958739684315989724461303909147339922191860658043521801, 9241788870093407465014997814935911097382500814801756935098155836397728106099523334246705961157671066245303333668970942281507336829575189728974388395659681831145940519102981007476528112682865812171672935880162049179615438710777169135173291088496238581425805761265979732869038558465547260502887415047337494394227076074014738871043749588806147418294141$.}
\subsection{Case $k=10$}
There are 34 primes:

\seqsplit{$6073, 3200209, 39436849, 73465393, 594797955822739276609, 819345902930686790609017, 18809263461615454339806361, 45383216669219147480607900121, 227554455001506706798175042857, 1433280611499591293572532623398241, 32902980177887110883846983237124881, 17339766918259845356777724068277462241, 360282554452088386416776048087744319102890881, 1205235064120434495355696332830370101214365435477624209, 7673882656136090525784960977384891983719281971373981755597657, 367068996074533139730258249731622265010452407297423877732898304522529, 13750998738022336874108420784900894868951194113427768738514537217027527263649, 45098525326058596839639651276281185272842101030467870318406043701475264728410841, 539143871619374455241608403260760879542876365010925742421882150562256255805612417, 22317436995605961571460624751928517784805949217364234406612743132036910991847837580803214521, 2571384827667589499088172531121089472068117545879409788661684733909122892816154423460059041389681, 90870591827933884565488966062603181997242564647235938399628061567228324236586343220872179409083510569, 237807963215862295716846765445726004356163675042030430910475537687241528749496935915804467206603610969, 2940288472048293303011340533492406825386718561461937982930101466182747543610278204042884494270601176032362521, 12201162304165511029027527954971070808311128484800839953669397737218918927045523829296612146787435833240176801, 1110825469467883963679088800381450036772638036769208452581637325828296435275148665715253672469297247890890731398473, 458744691219808314718636499696859708611251671508333688136001193669412436002655028061841642113522651733188015401876331225579399263939999377, 19887613045081775812243727519341400151279529490376959279891695202086573105423265361352332207758995991609436794966211034028960898464772898750932649, 1206069534311008523125649733964828709206505457217630561891719047679201235062444203902448900330659879552059703095632150723685086521542044097313485872929, 7959329410930411703009106189765855857001555288122677784776211864749892142765411830944273255432985331418938677024051121575496585318658233547749561367061134617498637045936633, 399666652292583791352652053129273280084703006879938221389846662842065920123038727918143831038146670803834059731820813663116322408681987239184704770699345631776462442125674409935093407430001628272526564367637545417469989714286860031453664929, 3059863688866294045584019258260237232276653963621009037054943177658198447701374353913795907322059843985705186493839087376474419640810654603582058540133766352939194136234653216971256974313497857224134612981401776706578672035211108254594241676190773406864248377, 289686265808612212138345695567607468220776206958450748752887425091354344002662849253203651930720267564711327064657466567322353733378801381229235289455282278839283253119219859004436264440891772135236753042222341879009834977642325596963427509822986578223231140729541909836474085885489316770217, 1564184613890177668936330949806464343245747896473280431401200775868402933692237961552059951608664415068108516196798238862912780828465364061183782850246615886498738530986794251334506062155059236360254324325261844327219808756481083643273462645099596892892252923843225951745612313146196525664509338081268193$.}
\end{document}